\documentclass[11pt]{article}

\usepackage{amssymb}
\usepackage{graphicx}
\usepackage{xcolor} 
\usepackage{tensor}
\usepackage{fullpage} 
\usepackage{amsmath}
\usepackage{amsthm}
\usepackage{verbatim}
\usepackage{hyperref}
\usepackage{mathtools}
\usepackage{esint}

\usepackage{cite}
\usepackage{array}

\definecolor{green}{rgb}{0,0.8,0} 

\newtheorem{definition}{Definition}
\newtheorem{lemma}{Lemma}
\newtheorem{corollary}{Corollary}
\newtheorem{remark}{Remark}
\newtheorem{theorem}{Theorem}
\newtheorem{prop}{Proposition}

\newcommand{\N}{\mathbb{N}}

\newcommand{\bR}{\mathbb{R}}

\newcommand{\bu}{\mathbf{u}}
\newcommand{\bv}{\mathbf{v}}

\newcommand{\norm}[1]{{\left\Vert #1 \right\Vert}}
\newcommand{\set}[1]{{\left\lbrace #1 \right\rbrace}}

\newcommand{\ep}{\varepsilon}
\newcommand{\bfvarphi}{\boldsymbol{\varphi}}
\newcommand{\bfphi}{\boldsymbol{\phi}}

\newcommand{\qand}{\quad \text{and} \quad}

\DeclareMathOperator*{\divg}{div}
\DeclareMathOperator*{\loc}{loc}
\DeclareMathOperator*{\bog}{Bog}

\begin{document}

\bibliographystyle{plain}

\title{Liouville-type theorems for steady solutions of the barotropic Navier--Stokes system}

\author{Youseung Cho\thanks{Department of Mathematics, Yonsei University, Seoul 03722, Republic of Korea. 
E-mail: youseung@yonsei.ac.kr} 
\and Eduard Feireisl \thanks{Czech Academy of Sciences, Institute
of Mathematics, \v{Z}itn\'a 25, 115 67 Praha 1, Czech Republic.
E-mail: feireisl@math.cas.cz}
\and 
Minsuk Yang\thanks{Department of Mathematics, Yonsei University, Seoul 03722, Republic of Korea. 
E-mail: m.yang@yonsei.ac.kr}
}

  \maketitle

\renewcommand{\thefootnote}{\fnsymbol{footnote}}
\renewcommand{\thefootnote}{\arabic{footnote}}

\begin{abstract} 
We establish various results concerning the uniqueness of zero velocity solutions 
for the static barotropic Navier--Stokes system. Some of them can be seen as  Liouville-type theorems for problems in unbounded physical space. 
\end{abstract}

\section{Introduction}
\label{S1}

We consider the compressible (barotropic) Navier--Stokes equations 
\begin{align}
\label{E11}
\divg(\rho \bu) &= 0, \\
\label{E12}
\divg(\rho \bu \otimes \bu) + \nabla p(\rho) - \divg \mathbb S(\nabla \bu) &= \rho \nabla G, 
\end{align}
describing the steady distribution of the mass density  
$\rho = \rho(x)$, and the velocity $\bu = \bu(x)$, in $\Omega \subset \bR^3$, of a barotropic viscous 
fluid subjected to a potential external force $\nabla G$. 
Here $p = p(\rho)$ denotes the pressure, and
$\mathbb S$ is the viscous stress given by Newton's rheological law
\begin{equation} \label{E12a}
\mathbb{S}(\nabla \bu) = \mu \left( \nabla \bu + \nabla^t \bu - \frac{2}{3} \divg \bu \mathbb{I} \right) + \eta \divg \bu \mathbb{I},\ \mu > 0, \ \eta \geq 0.
\end{equation}
In addition, we impose the no-slip boundary conditions
\begin{equation} \label{E12b}
\bu|_{\partial \Omega} = 0.	
\end{equation}
as soon as the spatial domain possesses a non-empty boundary $\partial \Omega$.

Problem \eqref{E11}--\eqref{E12b} admits an obvious (static) solution 
$\bu \equiv 0$, while 
\begin{equation} \label{E12d}
\nabla p(\rho) = \rho \nabla G.	
	\end{equation}
Our goal is to find some sufficient conditions for the static solution to be the unique \emph{stationary solution}, meaning there is no other 
(weak) solution with 
$\bu \ne 0$. Such a statement fits in the category of Liouville type theorems for the 
compressible Navier--Stokes system as soon as $\Omega =  \bR^3$.

There are results of Liouville type for the system in question obtained by 
several authors: \cite{MR2914143}, \cite{MR3223826}, \cite{MR3843580}, and \cite{MR4178952}. They focus primarily on regular solutions without 
the driving force $G$. Let us point out that stationary and even static solutions need not be regular, at least if $G \ne 0$ due to the inevitable 
presence of vacuum states, see e.g., \cite{FP7}. 

We show that the stationary velocity always vanishes under the physically relevant assumption that
the total mass of the fluid is finite; see Theorem \ref{T1} below. Besides, similar results hold for problems with finite energy; see Theorem \ref{T11}. 
We also consider the problem in the class of the so-called very weak solutions, where neither the total mass nor the total energy is not assumed to 
be controlled, Theorem \ref{T2}.
Finally, we consider various types of vanishing properties of the weak solution when $p(\rho) = \rho^{\gamma}$, Theorem \ref{T3}--\ref{T23}.

The paper is organized as follows. Section \ref{S2} introduces the necessary 
function spaces framework and contains some preliminary material 
on Liouville-type results for general elliptic problems. In Section \ref{m}, 
we introduce the concepts of weak and very weak solutions to the barotropic Navier--Stokes system and state our main results. Section \ref{pr} is devoted to the proofs. 

\section{Preliminaries}
\label{S2}

For tensor-valued function $F$ and $G$, we shall use the notation $F:G = \sum_{i,j} F_{ij} G_{ij}$.
We denote $\rho^{\ep} = \rho \ast \eta_{\ep}$, where $\eta_{\ep}$ is a standard mollifier.
We also denote $\varphi_{r,R} \in C^{\infty}_c(B(R))$ a non negative cut-off function satisfying $\varphi_{r,R} = 1$ on $B(r)$ and $|\nabla \varphi_{r,R}| \lesssim (R-r)^{-1}$.

We recall some definitions in \cite{MR2808162}.
We denote $D'(\Omega) = (C^{\infty}_c(\Omega))^{\ast}$. 
For $m \in \mathbb N$ and $1 \leq q < \infty$, the linear space
\begin{equation}
	\label{E21}
	D^{m,q}(\Omega) = \set{u \in L^1_{\loc}(\Omega) : D^lu \in L^q(\Omega), |l| = m}
\end{equation}
is equipped with the seminorm
\[
|u|_{m,q} := \left( \sum_{|l|=m} \int_{\Omega} |D^lu|^q \right)^{\frac{1}{q}}.
\]
Let $P_m$ be the class of all polynomials of degree $\leq m-1$.
For $u \in D^{m,q}$, set
\[
[u]_m = \set{w \in D^{m,q} : w = u + P \text{ for some } P \in P_m}.
\] 
Denote $\dot D^{m,q}(\Omega)$ the space of $[u]_m$ with $u \in D^{m,q}$.
Then \eqref{E21} induces the following norm in $\dot D^{m,q}$:
\[
|[u]_m|_{m,q} := |u|_{m,q}, \quad u \in [u]_m.
\]
Note that $\dot D^{m,2}$ is a Hilbert space.
We define the space $D^{m,q}_0(\Omega)$ as the completion of $C^{\infty}_c(\Omega)$ with respect to the norm $|\cdot|_{m,q}$.
As $|\cdot|_{m,q}$ defines a norm in $C^{\infty}_c(\Omega)$, $D^{m,q}_0(\Omega)$ is isomorphic to a closed subspace of $\dot D^{m,q}(\Omega)$.
We shall use another inner product in $D^{1,2}_0(\Omega; \bR^3)$.
For $\mu > 0 $ and $\eta \ge 0$,
\[
\langle \bu, \bv \rangle = \int_{\Omega} \mu ( \nabla \bu : \nabla \bv) + (\mu/3 + \eta)  \divg \bu \divg \bv dx.
\]
Then for $\bu \in D^{1,2}_0(\Omega; \bR^3)$ and $\bv \in C^{\infty}_{c}(\Omega; \bR^3)$, we have
\begin{equation}
	\label{E22}
	\langle \bu, \bv \rangle = \int_{\Omega} \mathbb S(\nabla \bu) : \nabla \bv dx. 
\end{equation}

Friedrichs' lemma about commutators
\begin{lemma}
	[Lemma 3.1 of \cite{MR2084891}]
	\label{L21}
	Suppose that $N \geq 2$.
	Let $1 \leq q,\beta \leq \infty$, $(q,\beta) \ne (1, \infty)$, $1/q + 1/\beta \leq 1$ and let
	\[
	\rho \in L^{\beta}_{\loc}(\bR^N), \quad \bu \in (W^{1,q}_{\loc}(\bR^N))^N.
	\]
	Then
	\[
	(\divg(\rho \bu) - \rho \divg \bu)^{\ep} - \bu \cdot \nabla \rho^{\ep} \to 0 
	\]
	strongly in $L^r_{\loc}(\bR^N)$ where $r \in [1,q)$ if $\beta = \infty$, $1 < q \leq \infty$ and $1/\beta + 1/q \leq 1/r \leq 1$ otherwise.
\end{lemma}

Next, we recall the following Liouville-type theorem for elliptic systems with constant coefficients.

\begin{lemma}
	[Theorem 1 of \cite{MR0709724}]
	\label{L2}
	Let $A_{\alpha}$ be a constant $M \times M$ matrix.
	Let $\bu \in C^{2l}(\bR^N, \mathbb C^M)$ satisfy for all $x \in \bR^N$,
	\begin{equation}
		\label{E23}
		L\bu(x) = \sum_{|\alpha| \leq 2l} A_{\alpha} \partial^{\alpha} \bu(x) = 0.
	\end{equation} 
	Assume that $\bu$ is bounded.
	Then $\bu$ is constant if and only if
	\[
	\sum_{|\alpha| \leq 2l} A_{\alpha} (ip)^{\alpha}
	\]
	is invertible for all $p \in \bR^N \setminus \set{0}$.
\end{lemma}

Consider linear elliptic equations in divergence form in $\Omega \subset \bR^3$
\begin{equation}
	\label{E24}
	D_j(a_{ij}(x) D_i \bu) = 0,
\end{equation}
where the coefficients $a_{ij}(x)$ are bounded measurable functions satisfying 
\[
\lambda |\xi|^2 \leq a_{ij}(x) \xi_i \xi_j \leq \Gamma |\xi|^2, \quad \forall x \in \Omega, \quad \xi \in \bR^3.
\]
We say that $\bu \in L_{\loc}^1(\Omega; \bR^3)$ is a very weak solution of \eqref{E24} if for all $\varphi \in C^{\infty}_c(\Omega)$,
\[
\int_{\Omega} \bu D_i(a_{ij} D_j \varphi) = 0.
\]

\begin{lemma}
	[Theorem 1.3 of \cite{MR2873863}]
	\label{L7}
	Assume that $a_{ij} \in C^{0,1}_{\loc}(\Omega)$.
	If $\bu \in L^1_{\loc}(\Omega; \bR^3)$ is a very weak solution of \eqref{E24}, then $\bu \in W^{2,p}_{\loc}(\Omega; \bR^3)$ for any $1 \leq p < \infty$.
\end{lemma}

\section{Main results}
\label{m}

We assume that the pressure satisfies 
\begin{equation}
\label{E13}
p \in C[0, \infty), \quad p(0)=0 \qand p(\rho) \geq 0
\end{equation}
for $\rho \geq 0$ and introduce the pressure potential $P$ satisfying
\begin{equation}
\label{E14}
P(0) = 0,\ 
P'(\rho) \rho - P(\rho) = p(\rho) \quad \text{ for } \rho > 0.
\end{equation}
Note that if $p \in C^1(0,\infty)$ and $p'(\rho) \ge 0$, then \[
P''(\rho) = \frac{p'(\rho)}{\rho} \ge 0 \quad \text{for all} \quad \rho > 0.
\]
In particular, $P:[0,\infty) \to \bR$ is convex and determined by \eqref{E14} modulo a linear function of $\rho$.
Therefore, we may fix $P$ by setting either 
\[
P'(0) = 0 \quad \text{if} \quad P'(0) > - \infty \quad \text{or} \quad P'(0) = - \infty.
\]

\subsection{Weak and very weak solutions}

\begin{definition} 
Let $\Omega \subset \bR^3$ be a domain.
We say that $(\rho,\bu)$ is \emph{weak solution} to \eqref{E11}-- \eqref{E12b} if 
\begin{enumerate}
\item
$\rho$, $\bu$ are measurable in $\Omega$ and $\rho \geq 0$ a.e. in $\Omega$;
\item
$\rho \bu \in L^1_{\loc} (\overline{\Omega}; \bR^3)$ and 
\begin{equation}
\label{E15}
\int_{\Omega} \rho \bu \cdot \nabla \phi dx = 0
\end{equation}
for any $\phi \in C^{\infty}_c(\bR^3)$;
\item
$\rho, \rho |\bu|^2, p(\rho), \rho \nabla G \in L^1_{\loc} (\Omega; \bR^3)$, 
$\bu \in D^{1,2}_0(\Omega; \bR^3)$, and
\begin{equation}
\label{E16}
\int_{\Omega} [\rho \bu \otimes \bu : \nabla \bfvarphi + p(\rho) \divg \bfvarphi] dx = 
\int_{\Omega} \mathbb S(\nabla \bu): \nabla \bfvarphi dx 
- \int_{\Omega} \rho \nabla G \cdot \bfvarphi dx
\end{equation}
for any $\bfvarphi \in C^1_c(\Omega; \bR^3)$.
\end{enumerate}
\end{definition}

\begin{definition}
Let $\Omega = \bR^3$.	
We say that $(\rho, \bu)$ is \emph{very weak solution} to \eqref{E11}-- \eqref{E12b} if 
\begin{enumerate}
	\item
	$\rho$, $\bu$ are measurable in $\bR^3$ and $\rho \geq 0$ a.e. in $\bR^3$;
	\item
	$\rho \bu \in L^1_{\loc} (\bR^3; \bR^3)$ and 
	\begin{equation}
		\label{E15a}
		\int_{\bR^3} \rho \bu \cdot \nabla \phi dx = 0
	\end{equation}
	for any $\phi \in C^{\infty}_c(\bR^3)$;
	\item
	$\rho, \rho |\bu|^2, p(\rho), \rho \nabla G \in L^1_{\loc} (\bR^3; \bR^3)$, 
	$\bu \in L^1_{\rm loc}(\bR^3; \bR^3)$, and
	\begin{equation}
		\label{E16a}
		\int_{\bR^3} [\rho \bu \otimes \bu : \nabla \bfvarphi + p(\rho) \divg \bfvarphi] dx = -
		 \int_{\bR^3} \divg \mathbb S(\nabla \bfvarphi )\cdot \bu dx 
		- \int_{\bR^3} \rho \nabla G \cdot \bfvarphi dx
	\end{equation}
	for any $\bfvarphi \in C^2_c(\bR^3; \bR^3)$.
\end{enumerate}
\end{definition}

\subsection{Main results}

We first show that the velocity associated to a stationary problem on \emph{arbitrary} spatial domain must vanish as long as the total mass of the fluid is finite.

\begin{theorem}
\label{T1}
Let $\Omega \subset \bR^3$ be a domain and let $(\rho, \bu)$ be a weak solution to problem \eqref{E11} -- \eqref{E12b}.
Suppose $p(\rho)$ satisfies \eqref{E13} and $p(\rho) \lesssim \sqrt \rho$ for $\rho < 1$ and $G$ satisfies either $G \in D_0^{1, q}(\Omega)$ for some $6/5 \le q < \infty$ or $G \in W^{1, \infty}(\Omega)$.
If 
\begin{equation}
\label{E18}
\rho \in L^1 \cap L^{\infty} (\Omega),
\end{equation}
then $\bu=0$ a.e. 
\end{theorem}

\begin{remark}

We can also prove Theorem \ref{T1} under the assumption that $p$ satisfies \eqref{E13} and {$p \in C^1[0,\infty)$} since this, 
{together with \eqref{E18}},  imply $p(\rho) \in L^1 \cap L^{\infty}(\Omega)$.

\end{remark}

Next, we introduce the total energy
\[
E(\rho, \bu) = \frac{1}{2} \rho |\bu|^2 + P(\rho),
\]
where $P$ is the pressure potential defined in \eqref{E14}.

We consider the case when the energy is finite.

\begin{theorem}
\label{T11}
Let $\Omega \subset \bR^3$ be a domain and let $(\rho, \bu)$ be a weak solution to \eqref{E11} and \eqref{E12}.
Suppose $p(\rho)$ satisfies \eqref{E13}, $p \in C^1(0,\infty)$, and $p'(\rho) \ge 0$ for all $\rho \ge 0$.
Assume further that
\begin{equation}
\label{E181}
\liminf_{\rho \to 0^{+}} \frac{P(\rho)}{p(\rho)} > 0, \quad \liminf_{\rho \to 0^{+}} \frac{p(\rho)}{\rho^{\gamma}} > 0
\end{equation}
for some $\gamma > 1$.
Suppose
\begin{equation}
\label{E182}
\rho \in L^{\infty}(\Omega), \quad \int_\Omega E(\rho, \bu) dx < \infty,
\end{equation}
and $G \in D_0^{1, q}(\Omega)$ for some $\frac{6}{5} \le q \le \frac{6\gamma}{5\gamma-6}$ if $\gamma > \frac{6}{5}$ and $G \in D_0^{1, q}(\Omega)$ for some $6/5 \le q < \infty$ or $G \in W^{1, \infty}(\Omega)$ if $1 < \gamma \le \frac{6}{5}$.
Then
\[
\bu = 0 \ \mbox{a.e. in}\ \Omega. 
\]
\end{theorem}

Next, we consider the very weak solutions defined on $\Omega = \bR^3$.

\begin{theorem}
\label{T2}
Let $\Omega = \bR^3$ and let $(\rho, \bu)$ be a very weak solution to \eqref{E11} and \eqref{E12}.
Suppose $p(\rho)$ satisfies \eqref{E13} and $p \in C^1(0,\infty)$ and $G$ satisfies either $G \in D_0^{1, q}(\bR^3)$ for some $3/2 < q < \infty$ or $G \in W^{1, \infty}(\bR^3)$ and 
\begin{equation}
\label{E19}
\rho \in L^1 \cap L^{\infty} (\bR^3) \qand \bu \in L^{\infty}(\bR^3; \bR^3).
\end{equation}
Then we have the following:
\begin{enumerate}
\item
If $G=0$, then $\bu$ is equal to a constant vector field $\bu_{\infty}$ a.e.
Moreover, either $\bu_{\infty} = 0$ or $\rho = 0$ a.e.
\item
If $\text{ess} \limsup_{|x| \to \infty} |\bu(x)| = 0$, then $\bu = 0$ a.e. 
\end{enumerate}
\end{theorem}

\begin{remark}
\begin{itemize}
\item
Recently, authors in \cite{MR2914143}, \cite{MR3223826}, \cite{MR3843580}, and \cite{MR4178952} have studied the Liouville-type theorems for \eqref{E11} and $\eqref{E12}$ in $\bR^3$ when $(\rho, \bu)$ is smooth, $G$ is constant, and $p(\rho) = a \rho^{\gamma}$ with $a > 0 $ and $\gamma > 1$.
In particular, \cite{MR2914143} proved the Liouville theorem when $\rho \in L^{\infty}(\bR^3)$, $\nabla \bu \in L^2(\bR^3; \bR^{3 \times 3})$, and $\bu \in L^{\frac{3}{2}}(\bR^3; \bR^3)$.
In \cite{MR3223826}, the authors proved the Liouville theorem when $\rho \in L^{\infty}(\bR^3)$, $\nabla \cdot \bu \in L^2(\bR^3)$, and $\bu \in L^{\frac{9}{2}}(\bR^3; \bR^3)$.
They also obtained the theorem when $\rho \in L^{\infty}(\bR^3)$, and $\bu \in L^{\frac{3}{2}} \cap L^{\frac{9}{2}}(\bR^3; \bR^3)$.
In \cite{MR3843580}, the authors proved the Liouville theorem when $\rho \in L^{\infty}(\bR^3)$, $\nabla \bu \in L^2(\bR^3; \bR^{3 \times 3})$, and $\bu \in L^{\frac{9}{2}, \infty}(\bR^3; \bR^3)$.
In \cite{MR4178952}, the authors proved the Liouville theorem when $\rho \in L^{\infty}(\bR^3)$, $\nabla \bu \in L^2(\bR^3; \bR^{3 \times 3})$, and $\bu \in L^{p,q}(\bR^3; \bR^3)$ for $3 < p < \frac{9}{2}$, $3 \leq q \leq \infty$ or $p=q=3$.
\end{itemize}
\end{remark}

Finally, for the isentropic case, where $p(\rho)$ is expressed as $\rho^{\gamma}$, $L^{\infty}$ conditions in Theorem \ref{T2} can be relaxed to some integrability conditions.
For simplicity, we shall denote
\[
f(\alpha, \gamma) = 
\begin{cases}
\alpha - \frac{3}{2\gamma} &\quad \text{if} \quad  \alpha \ge \frac{2\gamma-3}{\gamma(\gamma-1)}, \\
\frac{\alpha}{\gamma} + \frac{1}{2\gamma} - \frac{3}{\gamma^2} &\quad \text{if} \quad 0 \le \alpha < \frac{2\gamma-3}{\gamma(\gamma-1)},
\end{cases}
\]
if $\gamma \ge 6$ and
\[
f(\alpha, \gamma) = 
\begin{cases}
\alpha - \frac{3}{2\gamma} &\quad \text{if} \quad  \alpha \ge \frac{3}{2(\gamma-1)}, \\
\frac{\alpha}{\gamma} &\quad \text{if} \quad \frac{1}{2(\gamma-1)} \le \alpha < \frac{3}{2(\gamma-1)}, \\
\frac{\alpha}{\gamma^2} + \frac{1}{2\gamma^2} &\quad \text{if} \quad 0 \le \alpha < \frac{1}{2(\gamma-1)},
\end{cases}
\]
if $3 < \gamma < 6$ and
\[
f(\alpha, \gamma) = 
\begin{cases}
\alpha - \frac{1}{\gamma} &\quad \text{if} \quad  \alpha \ge \frac{3}{2(\gamma-1)}, \\
\frac{\alpha}{\gamma} + \frac{1}{2\gamma} &\quad \text{if} \quad \frac{1}{2(\gamma-1)} \le \alpha < \frac{3}{2(\gamma-1)}, .
\end{cases}
\]
if $\frac{3}{2} < \gamma < 3$.

Then, we define  
\begin{align}
\begin{split}
\label{E110}
g(\alpha, \beta, \gamma) 
=
\begin{cases}
\frac{\gamma (1-f(\alpha, \gamma))(6-\beta)}{3\beta (\gamma-1)} \quad & \text{if} \quad f(\alpha, \gamma) \le \frac{\gamma-3}{\gamma(3\gamma-5)}, \\
\frac{(1-\gamma f(\alpha, \gamma))(6-\beta)}{2\beta} \quad & \text{if} \quad f(\alpha, \gamma) > \frac{\gamma-3}{\gamma(3\gamma-5)}.
\end{cases}
\end{split}
\end{align}

\begin{theorem}
\label{T3}
Let $\Omega = \bR^3$ and let $(\rho, \bu)$ be a weak solution of \eqref{E11} and \eqref{E12} with $p(\rho) = \rho^{\gamma}$ and $\gamma \ge 3$.
Suppose $1 \le \beta \le 2$, $0 \le \alpha \le \frac{5}{2\gamma}$.
Then $\bu = 0$ a.e. if $G \in D_{0}^{1,\frac{6\gamma}{5\gamma-6}}(\bR^3)$, 
\begin{equation}
\label{E111}
\sup_{R >1} R^{-\alpha} \norm{\rho}_{L^{\gamma}(B(R))} < \infty
\qand 
\liminf_{R \to \infty} R^{-g(\alpha, \beta, \gamma)} \norm{\bu}_{L^{\beta}(B(2R) \setminus B(R))} < \infty.
\end{equation}
\end{theorem}

We can remove the assumption on $\bu$ by imposing an additional assumption on $\rho$.
We first consider the case $\gamma \ge 6$.

\begin{theorem}
\label{T21}
Let $\Omega = \bR^3$ and let $(\rho, \bu)$ be a weak solution to \eqref{E11} and \eqref{E12} with $p(\rho) = \rho^{\gamma}$ and $\gamma \ge 6$.
Suppose $G \in D_{0}^{1,q}(\bR^3)$ for some $\frac{6\gamma}{5\gamma-6} \le q \le \frac{3}{2}$.
Then $\bu = 0$ a.e. if
\begin{equation}
\label{E}
\rho \in L^{6} \cap L^{\gamma} (\bR^3),
\end{equation}
\end{theorem}

We also consider when total energy is finite.

\begin{theorem}
\label{T22}
Let $\Omega = \bR^3$ and let $(\rho, \bu)$ be a weak solution to \eqref{E11} and \eqref{E12} with $p(\rho) = \rho^{\gamma}$ and $\gamma \ge 6$.
Then $\bu = 0$ a.e. if $G \in D_{0}^{1,\frac{6\gamma}{5\gamma-6}}(\bR^3)$ and
\begin{equation}
\label{E112}
\int_{\bR^3} \left( \frac{1}{2} \rho |\bu|^2 + \frac{1}{\gamma-1} \rho^{\gamma} \right) dx < \infty.
\end{equation}
\end{theorem}

For the $3 \le \gamma < 6$ case, we can prove the following simple result.

\begin{theorem}
\label{T23}
Let $\Omega = \bR^3$ and let $(\rho, \bu)$ be a weak solution to \eqref{E11} and \eqref{E12} with $p(\rho) = \rho^{\gamma}$ and $3 \le \gamma < 6$.
Suppose $G \in D_{0}^{1,\frac{6\gamma}{5\gamma-6}}(\bR^3)$.
If $\rho \in L^{\gamma} (\bR^3)$, then $\bu = 0$ a.e. 
\end{theorem}

\section{Proofs}
\label{pr}

\subsection{Proof of Theorem \ref{T1}}
\label{S3}

Since $\bu \in L^6(\Omega; \bR^3)$ due to the Sobolev inequality, we have from \eqref{E18} that
\begin{equation}
\label{E31}
\rho \bu \in L^1 \cap L^6(\Omega; \bR^3) 
\qand \rho \bu \otimes \bu \in L^1 \cap L^3(\Omega; \bR^{3 \times 3}).
\end{equation}
In addition, we also have $\rho \nabla G \in L^1 \cap L^q (\Omega; \bR^3)$ if $G \in D_0^{1, q}(\Omega)$ for some $q \ge 6/5$ and $\rho \nabla G \in L^1 \cap L^\infty (\Omega; \bR^3)$ if $G \in W^{1, \infty}(\Omega)$.
In any case, we have
\begin{equation}
\label{E310}
\rho \nabla G \in L^1 \cap L^\frac{6}{5} (\Omega; \bR^3).
\end{equation}
Moreover, \eqref{E13}, \eqref{E18}, and the assumption $p(\rho) \lesssim \sqrt \rho$ for $\rho < 1$ guarantees that
\begin{equation}
\label{E31b}
p(\rho) \in L^2 \cap L^{\infty}(\Omega).
\end{equation}
Consider a sequence of function $p_n \in C_c(0,\infty)$ such that $0 \le p_n \nearrow p$.
Then the Diperna--Lions theory yields the so-called renormalized version of \eqref{E15}, see e.g. \cite{MR1022305} or \cite[Lemma 11.13]{MR2499296}:   
\begin{equation}
\label{E31c}
\int_{\bR^3} [ \rho B_n(\rho) \bu \cdot \nabla \phi - p_n(\rho) \divg \bu \phi ] dx = 0
\end{equation}
for any $\phi \in C^{\infty}_c(\bR^3)$ and 
\[
B_n(\rho) = B(1) + \int_1^{\rho} \frac{p_n(z)}{z^2} dz.
\]
Note that $B_n$ is continuous and bounded.
From $\bu \in D^{1,2}_0(\Omega; \bR^3)$, \eqref{E18}, and \eqref{E31b}, we have
\[
\rho B_n(\rho) \bu \in L^1(\bR^3; \bR^3) \qand  p_n(\rho) \divg \bu \in L^1(\bR^3).
\]
Hence, we have from \eqref{E31c} that
\[
\int_{\bR^3} p_n(\rho) \divg \bu dx = 0.
\]
By Lebesgue dominated convergence theorem, we obtain
\begin{equation}
\label{E32}
\int_{\Omega} p(\rho) \divg \bu dx = 0.
\end{equation}
Since $\bu \in D^{1,2}_0(\Omega; \bR^3)$, $\rho \bu \otimes \bu \in L^2(\Omega; \bR^{3 \times 3})$, $p(\rho) \in L^2(\Omega)$, and $\rho \nabla G \in L^{\frac{6}{5}}(\Omega; \bR^3)$, momentum balance \eqref{E16} can be extended to $\bfvarphi \in D^{1,2}_0(\Omega; \bR^3)$.
Therefore, we can insert $\bfvarphi = \bu$ in \eqref{E16} and use \eqref{E32} to get  
\[
\int_{\Omega} \rho \bu \cdot \nabla \left( \frac{1}{2} |\bu|^2 + G \right) dx 
= \int_{\Omega} \mathbb S(\nabla \bu) : \nabla \bu dx.
\]
As $\rho \bu \in L^1 \cap L^6(\Omega; \bR^3)$, mass conservation equation \eqref{E15} can be extended to $\phi \in W^{1, \infty}(\Omega; \bR^3)$ or $\phi \in D^{1,q}_0(\Omega; \bR^3)$ for $6/5 \le q < \infty$.
Therefore, we can conclude that
\begin{equation}
\label{E33}
\int_{\Omega} \rho \bu \cdot \nabla \left( \frac{1}{2} |\bu|^2 + G \right) dx = 0 = \int_{\Omega} \mathbb S(\nabla \bu) : \nabla \bu dx.
\end{equation}
In order to conclude $\nabla \bu = 0$, let $\bu_n \in C^{\infty}_c(\Omega)$ and $\bu_n \to \bu$ in $D^{1,2}_0(\Omega; \bR^3)$. 
Using integration by parts, we have
\begin{align*}
\int_{\Omega} \nabla^t \bu : \nabla \bu dx 
&= \int_{\Omega} \left( \partial_j \bu_i - \partial_j (\bu_n)_i \right) \left( \partial_i \bu_j - \partial_i (\bu_n)_j \right) dx \\
&\quad + \int_{\Omega} \nabla^t \bu_n : \nabla \bu + \nabla \bu_n : \nabla^t \bu - |\divg \bu_n|^2 dx.
\end{align*}
Letting $n \to \infty$, we have
\[
\int_{\Omega} \nabla^t \bu : \nabla \bu dx = \int_{\Omega} |\divg \bu|^2 dx.
\]
From \eqref{E33} and the definition $\mathbb S(\nabla \bu) = \mu (\nabla \bu + \nabla^t \bu) + (\eta - 2\mu/3) \divg \bu \mathbb I$, we get
\[
\int_{\Omega} \mu |\nabla \bu|^2 + (\mu/3 + \eta) |\divg \bu|^2 dx = 0
\]
and hence $\nabla \bu = 0$ a.e. 
Since $\bu \in D^{1,2}_0(\Omega;\bR^3)$, we can conclude that $\bu = 0$ a.e. 
This completes the proof of Theorem \ref{T1}.

\subsection{Proof of Theorem \ref{T11}}
\label{S31}

First, we observe that hypothesis \eqref{E181}, along with uniform boundedness of the density, imply
\[
\int_{\Omega} \rho^{\gamma} dx \lesssim 
\int_{\Omega} p(\rho) dx \lesssim
\int_{\Omega} P(\rho) dx.
\]
That is  
\[
\rho \in L^{\gamma} \cap L^{\infty}(\Omega).
\]
Combining this with \eqref{E182} and $\bu \in D^{1,2}_0(\Omega; \bR^{3})$, we have
\[
\rho \bu \in L^{q(\gamma)} \cap L^6(\Omega; \bR^3)
\]
where $q(\gamma) = \frac{2\gamma}{\gamma+1}$ if $\gamma \ge \frac{3}{2}$, $q(\gamma) = \frac{6\gamma}{\gamma+6}$ if $\frac{6}{5} \le \gamma < \frac{3}{2}$, and $q(\gamma) = 1$ if $1 < \gamma < \frac{6}{5}$. 
Moreover, we also have
\[
\rho \bu \otimes \bu \in L^1 \cap L^3 (\Omega; \bR^{3 \times 3}),
\qand p(\rho), P(\rho) \in L^1 \cap L^{\infty} (\Omega).
\]
Note that the assumption on $G$ guarantees
\[
\rho \nabla G \in L^{\frac{6}{5}}(\Omega; \bR^{3}) \qand \rho \bu \cdot \nabla G \in L^1(\Omega).
\]
Then, momentum balance \eqref{E16} can be extended to $\bfvarphi \in D^{1,2}_0(\Omega; \bR^3)$.
Therefore, we can insert $\bfvarphi = \bu$ in \eqref{E16} to get  
\[
\int_{\Omega} \rho \bu \cdot \nabla \left( \frac{1}{2} |\bu|^2 + G \right) dx 
= \int_{\Omega} \mathbb S(\nabla \bu) : \nabla \bu dx.
\]
Here, note that \eqref{E14} and the Diperna--Lions theory yields, see e.g. \cite[Lemma 3.3]{MR2084891}:   
\[
\int_{\bR^3} [P(\rho) \bu \cdot \nabla \phi - p(\rho) \divg \bu \phi] dx = 0 
\]
for any $\phi \in C^{\infty}_c(\bR^3)$.
Therefore, we have
\[
\int_{\Omega} p(\rho) \divg \bu = 0
\]
since $P(\rho) \bu \in L^1(\Omega; \bR^3)$ and $p(\rho) \divg \bu \in L^1(\Omega)$.
As in the proof of Theorem \ref{T1}, we can conclude that $\bu = 0$ a.e.

\subsection{Proof of Theorem \ref{T2}}
\label{S4} 

From the assumption \eqref{E19}, \eqref{E13}, and $p \in C^1(0,\infty)$, we have
\[
\rho \bu \in L^1 \cap L^\infty(\bR^3; \bR^3), 
\quad \rho \bu \otimes \bu \in L^1 \cap L^\infty(\bR^3; \bR^{3 \times 3}), 
\qand p(\rho) \in L^1 \cap L^\infty(\bR^3).
\] 
Moreover, we also have
\[
 \rho \nabla G \in L^1 \cap L^{q}(\bR^3; \bR^3).
\]
with $q > \frac{3}{2}$.
Consequently, the mapping
\[
\bfvarphi \mapsto \int_{\bR^3} [\rho \bu \otimes \bu : \nabla \bfvarphi + p(\rho) \divg \bfvarphi] dx + \int_{\bR^3} \rho \nabla G \cdot \bfvarphi dx
\]
represents a bounded linear functional on the Hilbert space $D^{1,2}_0(\bR^3; \bR^3)$.
By virtue of Riesz representation theorem and \eqref{E22}, there is a unique $\tilde \bu \in D^{1,2}_0(\bR^3;\bR^3)$ such that 
\begin{equation}
\label{E41}
\int_{\bR^3} \mathbb S(\nabla \tilde \bu) : \nabla \bfvarphi dx 
= \langle \tilde \bu, \bfvarphi \rangle 
= \int_{\bR^3} [\rho \bu \otimes \bu : \nabla \bfvarphi + p(\rho) \divg \bfvarphi] dx
+ \int_{\bR^3} \rho \nabla G \cdot \bfvarphi dx
\end{equation}
for all $\varphi \in C^{\infty}_c(\bR^3; \bR^3)$.
Since \eqref{E41} can be extended to $\bfvarphi \in D^{1,2}_0(\bR^3;\bR^3)$, we get the energy balance by inserting $\bfvarphi = \tilde \bu$;
\begin{equation}
\label{E42}
\int_{\bR^3} \mathbb S(\nabla \tilde \bu) : \nabla \tilde \bu dx
= \int_{\bR^3} [\rho \bu \otimes \bu : \nabla \tilde \bu + p(\rho) \divg \tilde \bu] dx
+ \int_{\bR^3} \rho \nabla G \cdot \tilde \bu dx.
\end{equation}
Moreover, using the standard elliptic $L^p$ estimates, we have
\[
\nabla \tilde \bu \in L^{q_1}(\bR^3; \bR^{3 \times 3}) \qand \tilde \bu \in L^{q_2}(\bR^3; \bR^3) 
\]
for $3/2 < q_1 \le \frac{3q}{3-q}$ if $q < 3$ and $3/2 < q_1 < \infty$ if $q \ge 3$ and $3 < q_2 \le \infty$.
Hence, $\tilde \bu$ is bounded and continuous on $\bR^3$.
Combining \eqref{E16a} and \eqref{E41}, we have $\bu - \tilde \bu \in L^{\infty}(\bR^3; \bR^3)$ and
\[
\int_{\bR^3} (\bu - \tilde \bu) \cdot \divg \mathbb S(\nabla \bfvarphi) dx = 0
\]
for all $\bfvarphi \in C^{\infty}_c (\bR^3; \bR^3)$.
By Lemma \ref{L7} and the local elliptic regularity theorem, we have that $\bu - \tilde \bu$ is a smooth solution of the Lam\'e system.
Applying Lemma \ref{L2}, we conclude that $\bu - \tilde \bu$ is a constant vector.
\begin{enumerate}
\item
Assume $G=0$.

Since $G=0$ and $\bu - \tilde \bu$ is a constant vector, we get from \eqref{E42} that
\[
\int_{\bR^3} \mathbb S(\nabla \bu) : \nabla \bu dx
= \int_{\bR^3} [\rho \bu \otimes \bu : \nabla \bu + p(\rho) \divg \bu] dx.
\]
As $\bu \in L^{\infty}(\bR^3; \bR^3)$ and $\nabla \bu \in L^2(\bR^3; \bR^{3 \times 3})$, we can repeat the same argument in Theorem \ref{T1} to get
\[
\int_{\bR^3} p(\rho) \divg \bu = 0.
\]
So we have
\begin{align*}
\int_{\bR^3} \mathbb S(\nabla \bu) : \nabla \bu dx
= \frac{1}{2} \int_{\bR^3} \rho \bu \cdot \nabla |\bu|^2 dx
\end{align*}
and from mass conservation equation \eqref{E15a},
\[
\int_{\bR^3} \mathbb S(\nabla \bu) : \nabla \bu dx = 0
\]
Therefore, $\bu$ equals a constant vector field $\bu_{\infty}$ a.e.
Now suppose $\bu_{\infty} \ne 0$.
Inserting $\bu = \bu_{\infty}$ to \eqref{E15}, we have
\[
\int \rho \bu_{\infty} \cdot \nabla \phi dx = 0
\]
for all $\phi \in C^{\infty}_c(\bR^3)$.
Then
\[
0 = \int \rho \bu_{\infty} \cdot \nabla \phi^{\ep} dx = - \int \bu_{\infty} \cdot \nabla \rho^{\ep} \phi dx
\]
for all $\phi \in C^{\infty}_c(\bR^3)$.
Then directional derivative of $\rho^{\ep}$ along $\bu_{\infty}$ direction is $0$.
From the assumption $\rho \in L^1(\bR^3)$, we have $\rho^{\ep} = 0$ for all $\ep > 0$.
Since $\rho^{\ep} \to \rho$ in $L^1(\bR^3)$, we can conclude that $\rho = 0$ a.e.

\item
Assume $\text{ess} \limsup_{|x| \to \infty} |\bu(x)| = 0$.

Since $\bu - \tilde \bu$ is a constant vector, we have
\[
\bu = \tilde \bu.
\]
That is, $\bu$ also belongs to $D^{1,2}_0(\bR^3;\bR^3)$. 
From \eqref{E42}, we have
\[
\int_{\bR^3} \mathbb S(\nabla \bu) : \nabla \bu dx
= \int_{\bR^3} [\rho \bu \otimes \bu : \nabla \bu + p(\rho) \divg \bu] dx
+ \int_{\bR^3} \rho \nabla G \cdot \bu dx.
\]
Similarly, we can easily obtain 
\[
\int_{\bR^3} \mathbb S(\nabla \bu) : \nabla \bu dx
= \frac{1}{2} \int_{\bR^3} \rho \bu \cdot \nabla |\bu|^2 dx
+ \int_{\bR^3} \rho \bu \cdot \nabla G dx.
\]
Therefore, we can conclude from mass conservation equation \eqref{E15a} that 
\[
\int_{\bR^3} \mathbb S(\nabla \bu) : \nabla \bu dx = 0.
\]
Since $\bu \in D^{1,2}_0(\bR^3; \bR^3)$, we obtain $\bu=0$ a.e.
\end{enumerate}

This completes the proof of Theorem \ref{T2}.
\qed

\subsection{Higher integrability of density}
\label{S5}

\begin{prop}
\label{P1}
Let $(\rho, \bu)$ be a weak solution of \eqref{E11} and \eqref{E12} in $\bR^3$ with $p(\rho) = \rho^{\gamma}$ with $\gamma \ge 6$ and $G \in D_{0}^{1,q}(\bR^3)$ for some $\frac{6\gamma}{5\gamma-6} \le q \le \frac{3}{2}$.
If 
\begin{equation}
\label{E51}
\sup_{R > 1} R^{-\alpha} \norm{\rho}_{L^{\gamma}(B(R))} < \infty
\end{equation}
for some $0 \le \alpha \le \frac{3}{\gamma}$, then we have the followings:
\begin{itemize}
\item
If $\frac{2\gamma-3}{\gamma(\gamma-1)} \le \alpha \le \frac{3}{\gamma}$, then
\begin{equation}
\label{E52}
\sup_{R > 1} R^{- \left( \alpha - \frac{3}{2\gamma} \right)} \norm{\rho}_{L^{2\gamma}(B(R))} < \infty.
\end{equation}
\item
If $0 \le \alpha < \frac{2\gamma-3}{\gamma(\gamma-1)}$, then
\begin{equation}
\label{E52b}
\sup_{R > 1} R^{- \left( \frac{\alpha}{\gamma} + \frac{1}{2\gamma} - \frac{3}{\gamma^2} \right)} \norm{\rho}_{L^{2\gamma}(B(R))} < \infty.
\end{equation}
\end{itemize}
\end{prop}

\begin{proof}
For $n \ge 1$, we shall denote
\[
R_n = \left( 1 + \frac{1}{2^n}\right)R \qand \gamma_{n} = \left( 1 - \frac{1}{2^{n}} \right) \gamma.
\]
Define the function
\[
\bfphi_{n} = \bog \left( \left(\rho^{\gamma_{n}}\right)^{\ep} \varphi_{R_n, R_{n-1}} 
- \fint_{B(R_{n-1})} \left(\rho^{\gamma_n}\right)^{\ep} \varphi_{R_n, R_{n-1}} dx \right).
\]
From the properties of the Bogovskii operator, we have
\begin{equation}
\label{E53}
\divg \bfphi_{n} = \left(\rho^{\gamma_n}\right)^{\ep} \varphi_{R_n, R_{n-1}} 
- \fint_{B(R_{n-1})} \left(\rho^{\gamma_n}\right)^{\ep} \varphi_{R_n, R_{n-1}} dx
\end{equation}
and 
\begin{equation}
\label{E54}
\norm{\nabla \bfphi_{n}}_{L^p(B(R_{n-1}))} 
\lesssim_p \norm{\left(\rho^{\gamma_n}\right)^{\ep}}_{L^p\left( B(R_{n-1})\right)} 
\lesssim \norm{\rho}_{L^{p \gamma_n}\left(B(R_{n-1})\right)}^{\gamma_n}
\end{equation}
for all $1 < p < \infty$.
As $\bfphi_{n} \in C^{\infty}_c(B(R_{n-1}); \bR^3)$, we have
\[
\int_{B(R_{n-1})} [\rho \bu \otimes \bu : \nabla \bfphi_{n} + \rho^\gamma \divg \bfphi_{n}] dx 
= \int_{B(R_{n-1})} \mathbb S(\nabla \bu) : \nabla \bfphi_{n} dx
-\int_{B(R_{n-1})} \rho \nabla G \cdot \bfphi_n dx
\]
by inserting $\bfvarphi = \bfphi_{n}$ in \eqref{E16}.
By \eqref{E53}, we have
\begin{align*}
\int_{B(R_{n-1})} \rho^\gamma \divg \bfphi_{n} dx 
&= \int_{B(R_{n-1})} \rho^\gamma \left(\rho^{\gamma_n}\right)^{\ep} \varphi_{R_n, R_{n-1}} dx \\ 
&\quad - \int_{B(R_{n-1})} \rho^\gamma dx \fint_{B(R_{n-1})} \left(\rho^{\gamma_n}\right)^{\ep} \varphi_{R_n, R_{n-1}} dx. 
\end{align*}
Therefore, we have
\begin{align*}
\int_{B(R_{n-1})} \rho^\gamma \left(\rho^{\gamma_n}\right)^{\ep} \varphi_{R_n, R_{n-1}} dx
&= \int_{B(R_{n-1})} \rho^\gamma dx \fint_{B(R_{n-1})} \left(\rho^{\gamma_n}\right)^{\ep} \varphi_{R_n, R_{n-1}} dx \\
&\quad + \int_{B(R_{n-1})} \mathbb S(\nabla \bu) : \nabla  \bfphi_{n} dx 
- \int_{B(R_{n-1})} \rho \bu \otimes \bu : \nabla \bfphi_{n} dx \\
&\quad - \int_{B(R_{n-1})} \rho \nabla G \cdot \bfphi_n dx \\
&:= K_1 + K_2 + K_3 + K_4. 
\end{align*}
Using \eqref{E54} and Jensen's inequality, we have
\[
\fint_{B(R_{n-1})} \left(\rho^{\gamma_n}\right)^{\ep} \varphi_{R_n, R_{n-1}} dx
\lesssim R^{-\frac{3}{2}} \norm{\rho}_{L^{2\gamma_n}\left(B(R_{n-1})\right)}^{\gamma_n}
\]
and hence
\[
K_1 \lesssim R^{-\frac{3}{2} + \alpha \gamma} \norm{\rho}_{L^{2\gamma_n}\left(B(R_{n-1})\right)}^{\gamma_n}.
\]
Since $\gamma \geq 6$, we have from \eqref{E51}, \eqref{E54}, H\"older's inequality and Sobolev's inequality that
\[
K_2 
\lesssim\norm{\nabla \bu}_{L^2(\bR^3)} 
\norm{\nabla \bfphi_n}_{L^{2}(B(R_{n-1}))}
\lesssim \norm{\nabla \bu}_{L^2(\bR^3)} 
\norm{\rho}_{L^{2\gamma_n} \left(B(R_{n-1})\right)}^{\gamma_n}
\]
and
\begin{align*}
K_3
&\lesssim R^{\frac{1}{2} - \frac{3}{\gamma}}
\norm{\rho}_{L^{\gamma}(B(R_{n-1}))}
\norm{\bu}_{L^{6}(\bR^3)}^2
\norm{\nabla \bfphi_n}_{L^{2}(B(R_{n-1}))} \\
&\lesssim R^{\alpha + \frac{1}{2} - \frac{3}{\gamma}} \norm{\nabla \bu}_{L^2(\bR^3)}^2 
\norm{\rho}_{L^{2\gamma_n}(B(R_{n-1}))}^{\gamma_n}.
\end{align*}
By H\"older's inequality and Sobolev's inequality, we have
\[
K_4 \lesssim \norm{\rho}_{L^{\frac{6q}{5q-6}}(B(R_{n-1}))} \norm{\nabla G}_{L^{q}(\bR^3)} \norm{\nabla \bfphi_n}_{L^{2}(B(R_{n-1}))}.
\]
Then from \eqref{E51} and \eqref{E54},
\[
K_4 \lesssim R^{\frac{5}{2} - \frac{3}{q} - \frac{3}{\gamma}} \norm{\rho}_{L^{\gamma}(B(R_{n-1}))} \norm{\nabla G}_{L^{q}(\bR^3)} \norm{\rho}_{L^{2\gamma_n} \left(B(R_{n-1})\right)}^{\gamma_n}
\lesssim R^{\alpha + \frac{1}{2} - \frac{3}{\gamma}} \norm{\nabla G}_{L^{q}(\bR^3)} \norm{\rho}_{L^{2\gamma_n} \left(B(R_{n-1})\right)}^{\gamma_n}.
\]
Using the assumption on $G$ and $\bu \in D^{1,2}_0(\bR^3; \bR^3)$, we have the following estimate.
\begin{equation}
\label{E55}
\int_{B(R_{n})} \rho^{\gamma} \left(\rho^{\gamma_n}\right)^{\ep} dx 
\lesssim \left( R^{-\frac{3}{2} + \alpha \gamma} + R^{\alpha + \frac{1}{2} - \frac{3}{\gamma}} \right) \norm{\rho}_{L^{2\gamma_n}\left(B(R_{n-1})\right)}^{\gamma_n}.
\end{equation}
We first consider the case $\alpha \ge \frac{2\gamma-3}{\gamma(\gamma-1)}$.
Since $\rho \geq 0$ a.e., we get the following by taking $\liminf_{\ep \to 0}$ on \eqref{E55};
\begin{equation}
\label{E56}
\int_{B(R_{n})} \rho^{2\gamma_{n+1}} dx  
\lesssim R^{-\frac{3}{2} + \alpha \gamma}
\norm{\rho}_{L^{2\gamma_n}\left(B(R_{n-1})\right)}^{\gamma_n}.
\end{equation}
Let $a_n$ be the minimum number satisfying
\[
\norm{\rho}_{L^{2\gamma_{n}}(B(R_{n-1}))} \lesssim R^{a_n}.
\]
Then $a_1=\alpha$ and from \eqref{E56}, we have 
\begin{equation}
\label{E57}
2\gamma_{n+1} a_{n+1} = \alpha \gamma - \frac{3}{2} + \gamma_n a_n.
\end{equation}
If we denote $b_n = (2^n-1) a_n \gamma$, then \eqref{E57} turns into 
\[
b_{n+1} - b_n = 2^n \left( \alpha \gamma - \frac{3}{2} \right).
\]
Therefore, we have 
\[
a_{n+1} =\alpha - \frac{3(2^n-1)}{(2^{n+1}-1)\gamma}
\]
and
\begin{equation}
\label{E58}
\norm{\rho}_{L^{2\gamma_{n+1}}(B(R))} \lesssim R^{\alpha - \frac{3(2^n-1)}{(2^{n+1}-1)\gamma}}.
\end{equation}
Taking $\liminf_{n \to \infty}$ both sides on \eqref{E58}, we obtain
\[
\sup_{R > 1} R^{- \left( \alpha - \frac{3}{2\gamma} \right)} \norm{\rho}_{L^{2\gamma}(B(R))} < \infty
\]
if $\frac{2\gamma-3}{\gamma(\gamma-1)} \le \alpha \le \frac{3}{\gamma}$.
If $0 \le \alpha < \frac{2\gamma-3}{\gamma(\gamma-1)}$, \eqref{E56} becomes
\[
\int_{B(R_{n})} \rho^{2\gamma_{n+1}} dx  
\lesssim R^{\alpha + \frac{1}{2} - \frac{3}{\gamma}}
\norm{\rho}_{L^{2\gamma_n}\left(B(R_{n-1})\right)}^{\gamma_n}.
\]
Repeating the same argument, we can easily obtain
\[
\sup_{R > 1} R^{- \left( \frac{\alpha}{\gamma} + \frac{1}{2\gamma} - \frac{3}{\gamma^2} \right)} \norm{\rho}_{L^{2\gamma}(B(R))} < \infty.
\]
This completes the proof of Proposition \ref{P1}.
\end{proof}

As a result of Proposition \ref{P1}, we can obtain the following version of higher integrability of $\rho$.

\begin{corollary}
\label{C1}
Let $(\rho, \bu)$ be a weak solution of \eqref{E11} and \eqref{E12} in $\bR^3$ with $p(\rho) = \rho^{\gamma}$ with $\gamma \ge 6$ and $G \in D_{0}^{1,q}(\bR^3)$ for some $\frac{6\gamma}{5\gamma-6} \le q \le \frac{3}{2}$.
If $\rho \in L^6(\bR^3) \cap L^{\gamma}(\bR^3)$, then $\rho \in L^{2\gamma}(\bR^3)$.
\end{corollary}

\begin{proof}
Recall from Proposition \ref{P1} that 
\begin{align*}
\int_{B(R_{n-1})} \rho^\gamma \left(\rho^{\gamma_n}\right)^{\ep} \varphi_{R_n, R_{n-1}} dx
&= \int_{B(R_{n-1})} \rho^\gamma dx \fint_{B(R_{n-1})} \left(\rho^{\gamma_n}\right)^{\ep} \varphi_{R_n, R_{n-1}} dx \\
&\quad + \int_{B(R_{n-1})} \mathbb S(\nabla \bu) : \nabla  \bfphi_{n} dx 
- \int_{B(R_{n-1})} \rho \bu \otimes \bu : \nabla \bfphi_{n} dx \\
&\quad - \int_{B(R_{n-1})} \rho \nabla G \cdot \bfphi_n dx \\
&:= K_1 + K_2 + K_3 + K_4. 
\end{align*}
This time, we shall estimate each term using 
\[
\sup_{R>0} \left( \norm{\rho}_{L^6(B(R))} + \norm{\rho}_{L^{\gamma}(B(R))} \right) < \infty. 
\]
As an analogue of Proposition \ref{P1}, we can easily obtain
\begin{align*}
K_1 &\lesssim R^{-\frac{3}{2}} \norm{\rho}_{L^{\gamma}\left(B(R_{n-1})\right)}^{\gamma} \norm{\rho}_{L^{2\gamma_n}\left(B(R_{n-1})\right)}^{\gamma_n}
\lesssim \norm{\rho}_{L^{2\gamma_n}\left(B(R_{n-1})\right)}^{\gamma_n}, \\
K_2 &\lesssim \norm{\nabla \bu}_{L^2(\bR^3)} \norm{\rho}_{L^{2\gamma_n}\left(B(R_{n-1})\right)}^{\gamma_n} \lesssim \norm{\rho}_{L^{2\gamma_n}\left(B(R_{n-1})\right)}^{\gamma_n},\\
K_4 &\lesssim \norm{\rho}_{L^{\frac{6q}{5q-6}}(B(R_{n-1}))} \norm{\nabla G}_{L^{q}(\bR^3)} \norm{\rho}_{L^{2\gamma_n}\left(B(R_{n-1})\right)}^{\gamma_n} \lesssim \norm{\rho}_{L^{2\gamma_n}\left(B(R_{n-1})\right)}^{\gamma_n}.
\end{align*}
Note that $6 \le \frac{6q}{5q-6} \le \gamma$.
For $K_3$, we use H\"older's inequality to get
\[
K_3 \lesssim 
\norm{\rho}_{L^{6}(B(R_{n-1}))}
\norm{\bu}_{L^{6}(\bR^3)}^2
\norm{\nabla \bfphi_n}_{L^{2}(B(R_{n-1}))} 
\lesssim 
\norm{\rho}_{L^{2\gamma_n}(B(R_{n-1}))}^{\gamma_n}.
\]
Hence we have $a_1 = 0$ and
\[
2 \gamma_{n+1} a_{n+1} = \gamma_n a_n
\]
instead of \eqref{E57}.
Repeating the same argument as in Proposition \ref{P1}, we obtain
\[
\sup_{R > 1} \norm{\rho}_{L^{2\gamma}(B(R))} < \infty.
\]
\end{proof}

\begin{corollary}
\label{C11}
Let $\Omega = \bR^3$ and let $(\rho, \bu)$ be a weak solution to \eqref{E11} and \eqref{E12} with $p(\rho) = \rho^{\gamma}$ and $\gamma \ge 6$.
If $G \in D_{0}^{1,\frac{6\gamma}{5\gamma-6}}(\bR^3)$ 
and
\[
\int_{\bR^3} \left( \frac{1}{2} \rho |\bu|^2 + \frac{1}{\gamma-1} \rho^{\gamma} \right) dx < \infty,
\]
then $\rho \in L^{2\gamma}(\bR^3)$.
\end{corollary}

\begin{proof}
First note that $\rho \in L^{\gamma}(\bR^3)$ and
\[
\rho \bu \otimes \bu \in L^1 \cap L^{\frac{3\gamma}{\gamma+3}}(\bR^3; \bR^3 \times \bR^3).
\]
Since $\rho \bu \otimes \bu \in L^2(\bR^3; \bR^3 \times \bR^3)$, we can estimate $K_3$ as
\[
K_3 \lesssim 
\norm{\rho \bu \otimes \bu}_{L^{2}(\bR^3)}^2
\norm{\nabla \bfphi_n}_{L^{2}(B(R_{n-1}))} 
\lesssim 
\norm{\rho}_{L^{2\gamma_n}(B(R_{n-1}))}^{\gamma_n}.
\]
For $K_4$, we have
\[
K_4 \lesssim \norm{\rho}_{L^{\gamma}(B(R_{n-1}))} \norm{\nabla G}_{L^{\frac{6\gamma}{5\gamma-6}}(\bR^3)} \norm{\rho}_{L^{2\gamma_n}\left(B(R_{n-1})\right)}^{\gamma_n}
\lesssim \norm{\rho}_{L^{2\gamma_n}\left(B(R_{n-1})\right)}^{\gamma_n}.
\]
We can get the desired result by repeating the same argument for the other terms.
\end{proof}

\begin{prop}
\label{P2}
Let $(\rho, \bu)$ be a weak solution of \eqref{E11} and \eqref{E12} in $\bR^3$ with $p(\rho) = \rho^{\gamma}$ with $3 < \gamma < 6$ and $G \in D_{0}^{1,\frac{6\gamma}{5\gamma-6}}(\bR^3)$.
If 
\begin{equation}
\label{E59}
\sup_{R > 1} R^{-\alpha} \norm{\rho}_{L^{\gamma}(B(R))} < \infty
\end{equation}
for some 
$0 \le \alpha \le \frac{3}{\gamma}$, then we have the followings:
\begin{itemize}
\item
If $\frac{3}{2(\gamma-1)} \le \alpha \le \frac{3}{\gamma}$, then
\begin{equation}
\label{E510}
\sup_{R > 1} R^{- \left( \alpha - \frac{3}{2\gamma} \right)} \norm{\rho}_{L^{2\gamma}(B(R))} < \infty.
\end{equation}
\item
If $\frac{1}{2(\gamma-1)} \le \alpha < \frac{3}{2(\gamma-1)}$, then
\begin{equation}
\label{E511}
\sup_{R > 1} R^{-\frac{\alpha}{\gamma}} \norm{\rho}_{L^{2\gamma}(B(R))} < \infty.
\end{equation}
\item
If $0 \le \alpha < \frac{1}{2(\gamma-1)}$, then
\begin{equation}
\label{E511b}
\sup_{R > 1} R^{- \left( \frac{\alpha}{\gamma^2} + \frac{1}{2\gamma^2} \right)} \norm{\rho}_{L^{2\gamma}(B(R))} < \infty.
\end{equation}
\end{itemize}
\end{prop}

\begin{proof}
\textbf{Step 1)}
Fix $3 < \gamma < 6$.
Then there exists $k \in \N$ such that
\[
\frac{9 - 3 \cdot \left( \frac{2}{3} \right)^{k}}{3 - 2 \cdot \left( \frac{2}{3} \right)^k} \le \gamma < \frac{9 - 3 \cdot \left( \frac{2}{3} \right)^{k-1}}{3 - 2 \cdot \left( \frac{2}{3} \right)^{k-1}}.
\]
Consider a sequence defined by 
\[
\widetilde \gamma_n =  (2\gamma -3) \left( 1 - \left( \frac{2}{3} \right)^{n} \right).
\]
We also denote
\[
c_n = \gamma + \widetilde \gamma_n
\]
and $d_n$ be a (minimum) number satisfying 
\begin{equation}
\label{E512}
\norm{\rho}_{L^{c_n}(B(R_{n-1}))} \lesssim R^{d_n}
\end{equation}
where $R_n = \left( 1 + \frac{1}{2^n}\right)R$.
Note that $c_n < 6$ for $1 \le n \le k-1$, $c_k \ge 6$, and
\[
3 \widetilde \gamma_{n+1} = 2c_n-3.
\]
We shall first obtain $L^p$ estimate of $\rho$ for $p \ge 6$ as $c_k \ge 6$.
Define the function
\[
\widetilde \bfphi_{n} = \bog \left( \left(\rho^{\widetilde \gamma_{n}}\right)^{\ep} \varphi_{R_n, R_{n-1}} 
- \fint_{B(R_{n-1})} \left(\rho^{\widetilde \gamma_n}\right)^{\ep} \varphi_{R_n, R_{n-1}} dx \right).
\]
From the properties of the Bogovskii operator, we have
\begin{equation}
\label{E513}
\divg \widetilde \bfphi_{n} = \left(\rho^{\widetilde \gamma_n}\right)^{\ep} \varphi_{R_n, R_{n-1}} 
- \fint_{B(R_{n-1})} \left(\rho^{\widetilde \gamma_n}\right)^{\ep} \varphi_{R_n, R_{n-1}} dx
\end{equation}
and 
\begin{equation}
\label{E514}
\norm{\nabla \widetilde \bfphi_{n}}_{L^p(B(R_{n-1}))} 
\lesssim_p \norm{\left(\rho^{\widetilde \gamma_n}\right)^{\ep}}_{L^p\left( B(R_{n-1})\right)} 
\lesssim \norm{\rho}_{L^{p \widetilde \gamma_n}\left(B(R_{n-1})\right)}^{\widetilde \gamma_n}
\end{equation}
for all $1 < p < \infty$.
As $\widetilde \bfphi_{n} \in C^{\infty}_c(B(R_{n-1}); \bR^3)$, we have
\[
\int_{B(R_{n-1})} [\rho \bu \otimes \bu : \nabla \widetilde \bfphi_{n} + \rho^\gamma \divg \widetilde \bfphi_{n}] dx 
= \int_{B(R_{n-1})} \mathbb S(\nabla \bu) : \nabla \widetilde \bfphi_{n} dx
- \int_{B(R_{n-1})} \rho \nabla G \cdot \widetilde \bfphi_n dx
\]
by inserting $\bfvarphi = \widetilde \bfphi_{n}$ in \eqref{E16}.
By \eqref{E513}, we have
\begin{align*}
\int_{B(R_{n-1})} \rho^\gamma \divg \widetilde \bfphi_{n} dx 
&= \int_{B(R_{n-1})} \rho^\gamma \left(\rho^{\widetilde \gamma_n}\right)^{\ep} \varphi_{R_n, R_{n-1}} dx \\ 
&\quad - \int_{B(R_{n-1})} \rho^\gamma dx \fint_{B(R_{n-1})} \left(\rho^{\widetilde \gamma_n}\right)^{\ep} \varphi_{R_n, R_{n-1}} dx. 
\end{align*}
Hence,
\begin{align*}
\int_{B(R_{n-1})} \rho^\gamma \left(\rho^{\widetilde \gamma_n}\right)^{\ep} \varphi_{R_n, R_{n-1}} dx
&= \int_{B(R_{n-1})} \rho^\gamma dx \fint_{B(R_{n-1})} \left(\rho^{\widetilde \gamma_n}\right)^{\ep} \varphi_{R_n, R_{n-1}} dx \\
&\quad + \int_{B(R_{n-1})} \mathbb S(\nabla \bu) : \nabla \widetilde \bfphi_{n} dx 
- \int_{B(R_{n-1})} \rho \bu \otimes \bu : \nabla \widetilde \bfphi_{n} dx \\
&\quad - \int_{B(R_{n-1})} \rho \nabla G \cdot \widetilde \bfphi_n dx \\
&:= I_1 + I_2 + I_3 + I_4.  
\end{align*}

\textbf{Step 2)}
We first consider the case $\frac{3}{2(\gamma-1)} \le \alpha \le \frac{3}{\gamma}$.
For $1 \le n \le k$, we have the following estimates.
Using \eqref{E59}, \eqref{E514} and H\"older's inequality, 
\begin{align*}
I_1
&\lesssim R^{\alpha \gamma-\frac{3}{2}} \norm{\rho}_{L^{2\widetilde \gamma_n} \left(B(R_{n-1})\right)}^{\widetilde \gamma_n}, \\
I_2
&\lesssim \norm{\nabla \bu}_{L^2(\bR^3)} 
\norm{\rho}_{L^{2\widetilde \gamma_n} \left(B(R_{n-1})\right)}^{\widetilde \gamma_n}, \\
I_4 
&\lesssim R^{\alpha} \norm{\nabla G}_{L^{\frac{6\gamma}{5\gamma-6}}(\bR^3)} \norm{\rho}_{L^{2\widetilde \gamma_n} \left(B(R_{n-1})\right)}^{\widetilde \gamma_n}.
\end{align*}
From the assumption $G \in D_{0}^{1,\frac{6\gamma}{5\gamma-6}}(\bR^3)$, $\bu \in D^{1,2}_0(\bR^3; \bR^3)$ and $\alpha \ge \frac{3}{2(\gamma-1)}$, we have
\[
I_1 + I_2 + I_4
\lesssim R^{\alpha \gamma-\frac{3}{2}} \norm{\rho}_{L^{2\widetilde \gamma_n} \left(B(R_{n-1})\right)}^{\widetilde \gamma_n}.
\]
By Jensen's inequality and \eqref{E512}, we get
\begin{equation}
\label{E515}
I_1 + I_2 + I_4
\lesssim R^{\alpha \gamma-\frac{3}{2} + \frac{6-c_{n-1}}{2c_{n-1}}} \norm{\rho}_{L^{c_{n-1}} \left(B(R_{n-1})\right)}^{\widetilde \gamma_n}
\lesssim R^{\alpha \gamma - 2 + \frac{3}{c_{n-1}} + d_{n-1} \widetilde \gamma_n}.
\end{equation}
By H\"older's inequality and Sobolev's inequality, we have
\[
I_3 \lesssim \norm{\rho}_{L^{c_{n-1}} \left(B(R_{n-1})\right)} 
\norm{\nabla \bu}_{L^2(\bR^3)}^2 
\norm{\nabla \widetilde \bfphi_n}_{L^{\frac{3c_{n-1}}{2c_{n-1}-3}} \left(B(R_{n-1})\right)}. 
\]
From \eqref{E514}, \eqref{E512} and the assumption $\bu \in D^{1,2}_0(\bR^3; \bR^3)$, we obtain
\begin{equation}
\label{E516}
I_3 \lesssim \norm{\rho}_{L^{c_{n-1}} \left(B(R_{n-1})\right)}^{1 + \widetilde \gamma_n}
\lesssim R^{d_{n-1}(1+\widetilde \gamma_n)}.
\end{equation}
Combining \eqref{E515} and \eqref{E516},
\[
\int_{B(R_{n})} \rho^{\gamma} \left(\rho^{\widetilde \gamma_n}\right)^{\ep} dx 
\lesssim R^{\alpha \gamma - 2 + \frac{3}{c_{n-1}} + d_{n-1} \widetilde \gamma_n}
+ R^{d_{n-1}(1+\widetilde \gamma_n)}.
\]
If
\[
d_{n-1} \leq \alpha \gamma - 2 + \frac{3}{c_{n-1}},
\]
then we have
\begin{equation} 
\label{E517}
\int_{B(R_{n})} \rho^{\gamma} \left(\rho^{\widetilde \gamma_n}\right)^{\ep} dx 
\lesssim R^{\alpha \gamma - 2 + \frac{3}{c_{n-1}} + d_{n-1} \widetilde \gamma_n}.
\end{equation}
Since $\rho \geq 0$ a.e., we get the following by taking $\liminf_{\ep \to 0}$ on \eqref{E517};
\[
\int_{B(R_{n})} \rho^{c_n} dx  
\lesssim R^{\alpha \gamma - 2 + \frac{3}{c_{n-1}} + d_{n-1} \widetilde \gamma_n}.
\]
From this, we can derive a recurrence relation
\begin{equation}
\label{E518}
c_n d_n = \alpha \gamma - 2 + \frac{3}{c_{n-1}} + d_{n-1} \widetilde \gamma_n
\leq \widetilde \gamma_n d_{n-1} + \alpha \gamma - 1
\end{equation}
where we use $c_{n-1} \ge \gamma >3$ in the last inequality.
If we let $d_n = \alpha - \frac{1}{\gamma} + e_n$, \eqref{E518} becomes
\[
c_nd_n = (\gamma + \widetilde \gamma_n) e_n \le \widetilde \gamma_n e_{n-1}.
\]
Since we can easily compute 
\[
\frac{\widetilde \gamma_n}{\gamma + \widetilde \gamma_n} \le \frac{2}{3},
\]
we obtain
\[
d_n \le \frac{2}{3} d_{n-1} + \frac{1}{3} \left( \alpha - \frac{1}{\gamma} \right).
\]

If we assume
\[
d_{n-1} \ge \alpha \gamma - 2 + \frac{3}{c_{n-1}},
\]
then 
\[
\int_{B(R_{n})} \rho^{\gamma} \left(\rho^{\widetilde \gamma_n}\right)^{\ep} dx 
\lesssim R^{(1 + \widetilde \gamma_n) d_{n-1}}.
\]
instead of \eqref{E517}.
This time, we derive a recurrence relation 
\[
c_n d_n = (1 + \widetilde \gamma_n) d_{n-1}
\]
and using $\frac{1 + \widetilde \gamma_n}{\gamma + \widetilde \gamma_n} \le \frac{2}{3}$, we obtain
\[
d_n \le \frac{2}{3} d_{n-1}.
\]
Combining two relations, we get
\[
d_n \le \frac{2}{3} d_{n-1} + \frac{1}{3} \left( \alpha - \frac{1}{\gamma} \right)
\]
for $\frac{3}{2(\gamma-1)} \le \alpha \le \frac{3}{\gamma}$ case.
Therefore,
\[
d_n \le \alpha - \frac{1}{\gamma} + \frac{2^n}{3^n \gamma}.
\]
That is, 
\begin{equation}
\label{E519}
\norm{\rho}_{L^{c_n}(B(R_n))} \lesssim R^{\alpha - \frac{1}{\gamma} + \frac{2^n}{3^n \gamma}}
\end{equation}
for all $n \ge 0$.

\textbf{Step 3)}
For $0 \le \alpha < \frac{3}{2(\gamma-1)}$ case, we can obtain a similar estimate by performing the same argument. 
This time, we have
\[
I_1 + I_2 + I_4
\lesssim R^{\alpha} \norm{\rho}_{L^{2\widetilde \gamma_n} \left(B(R_{n-1})\right)}^{\widetilde \gamma_n}
\]
and hence
\[
\int_{B(R_{n})} \rho^{\gamma} \left(\rho^{\widetilde \gamma_n}\right)^{\ep} dx 
\lesssim R^{\alpha + \frac{3}{c_{n-1}} - \frac{1}{2} + d_{n-1} \widetilde \gamma_n}
+ R^{d_{n-1}(1+\widetilde \gamma_n)}.
\]
If 
\[
d_{n-1} \ge \alpha + \frac{3}{c_{n-1}} - \frac{1}{2}, 
\]
we get 
\[
d_n \le \frac{2}{3} d_{n-1}
\]
as in the previous step.

If 
\[
d_{n-1} \le \alpha + \frac{3}{c_{n-1}} - \frac{1}{2}, 
\]
then 
\[
c_n d_n = \alpha + \frac{3}{c_{n-1}} - \frac{1}{2} + d_{n-1} \widetilde \gamma_n
\le d_{n-1} \widetilde \gamma_n + \alpha + \frac{1}{2}.
\]
By defining a new sequence $e_n = d_n - \frac{\alpha}{\gamma} - \frac{1}{2\gamma}$, we get
\[
d_n \le \frac{2}{3} d_{n-1} + \frac{1}{3} \left( \frac{\alpha}{\gamma} + \frac{1}{2\gamma} \right).
\]
Combining two relations, we get 
\[
d_n \le \frac{\alpha}{\gamma} + \frac{1}{2\gamma} + \left( \frac{2}{3} \right)^{n} \left( \alpha - \frac{\alpha}{\gamma} - \frac{1}{2\gamma} \right)
\]
for $0 \le \alpha < \frac{3}{2(\gamma-1)}$ case.
That is, 
\begin{equation}
\label{E519b}
\norm{\rho}_{L^{c_n}(B(R_n))} \lesssim R^{\frac{\alpha}{\gamma} + \frac{1}{2\gamma} + \left( \frac{2}{3} \right)^{n} \left( \alpha - \frac{\alpha}{\gamma} - \frac{1}{2\gamma} \right)}
\end{equation}
for all $n \ge 0$.

\textbf{Step 4)}
In order to obtain desired result, recall from proposition \ref{P1} that $\gamma_n = \left( 1 - \frac{1}{2^{n}} \right) \gamma$, \[
\bfphi_{n} = \bog \left( \left(\rho^{\gamma_{n}}\right)^{\ep} \varphi_{R_n, R_{n-1}} 
- \fint_{B(R_{n-1})} \left(\rho^{\gamma_n}\right)^{\ep} \varphi_{R_n, R_{n-1}} dx \right),
\]
and
\begin{align*}
\int_{B(R_{n-1})} \rho^\gamma \left(\rho^{\gamma_n}\right)^{\ep} \varphi_{R_n, R_{n-1}} dx
&= \int_{B(R_{n-1})} \rho^\gamma dx \fint_{B(R_{n-1})} \left(\rho^{\gamma_n}\right)^{\ep} \varphi_{R_n, R_{n-1}} dx \\
&\quad + \int_{B(R_{n-1})} \mathbb S(\nabla \bu) : \nabla  \phi_{n} dx 
- \int_{B(R_{n-1})} \rho \bu \otimes \bu : \nabla \phi_{n} dx \\
&\quad - \int_{B(R_{n-1})} \rho \nabla G \cdot \bfphi_n dx \\
&:= K_1 + K_2 + K_3 + K_4. 
\end{align*}
First, we shall consider $\frac{3}{2(\gamma-1)} \le \alpha \le \frac{3}{\gamma}$ case.

Using the estimates in the proof of Proposition \ref{P1} with $q = \frac{6\gamma}{5\gamma-6}$, we have
\[
K_1 + K_2 + K_4 \lesssim R^{\alpha \gamma - \frac{3}{2}} \norm{\rho}_{L^{2\gamma_n}\left(B(R_{n-1})\right)}^{\gamma_n}
\]
and
\[
K_3 \lesssim \norm{\rho}_{L^{6}(B(R_{n-1}))}
\norm{\bu}_{L^6(\bR^3)}^2
\norm{\nabla \phi_n}_{L^{2}(B(R_{n-1}))}.
\]
From $3 < \gamma < 6$, $c_k \ge 6$, and \eqref{E519}, we can handle $L^6$ norm of $\rho$ as
\[
\norm{\rho}_{L^{6}(B(R_{n-1}))}
\lesssim \norm{\rho}_{L^{\gamma}(B(R_{n-1}))} + \norm{\rho}_{L^{c_k}(B(R_{n-1}))}
\lesssim R^{\alpha} + R^{\alpha - \frac{1}{\gamma} + \frac{2^k}{3^k \gamma}}
\lesssim R^{\alpha}.
\] 
Since $\frac{3}{2(\gamma-1)} \le \alpha$, we have
\[
\int_{B(R_{n})} \rho^{2\gamma_{n+1}} dx  
\lesssim R^{\alpha \gamma -\frac{3}{2} + \gamma_n a_n}
\]
as in Proposition \ref{P1}.
Therefore, we get \eqref{E58} and by taking $\liminf_{n \to \infty}$ both sides, we obtain \eqref{E510}.

\textbf{Step 5)}
If $\frac{1}{2(\gamma-1)} \le \alpha < \frac{3}{2(\gamma-1)}$, we have
\[
d_k \le \frac{\alpha}{\gamma} + \frac{1}{2\gamma} + \left( \frac{2}{3} \right)^{k} \left( \alpha - \frac{\alpha}{\gamma} - \frac{1}{2\gamma} \right)
\le \alpha
\]
and
\[
\norm{\rho}_{L^{6}(B(R_{n-1}))}
\lesssim R^{\alpha}
\]
so that
\[
\int_{B(R_{n})} \rho^{2\gamma_{n+1}} dx  
\lesssim R^{\alpha + \gamma_n a_n}.
\]
This leads to a recurrence relation
\[
2\gamma_{n+1} a_{n+1} = \alpha + \gamma_n a_n.
\]
Transforming to $b_n = (2^n-1) \gamma a_n$, we can easily obtain \eqref{E511}.

Lastly, if $0 \le \alpha < \frac{1}{2(\gamma-1)}$, we have 
\[
d_k
\le \frac{2}{3} \alpha + \frac{1}{3} \left( \frac{\alpha}{\gamma} + \frac{1}{2\gamma} \right)
\le \frac{\alpha}{\gamma} + \frac{1}{2\gamma}
\]
and
\[
\norm{\rho}_{L^{6}(B(R_{n-1}))}
\lesssim R^{\frac{\alpha}{\gamma} + \frac{1}{2\gamma}}.
\]
This time, we obtain
\[
2\gamma_{n+1} a_{n+1} = \frac{\alpha}{\gamma} + \frac{1}{2\gamma} + \gamma_n a_n
\]
and this leads to \eqref{E511b}.
\end{proof}

As a simple result of Proposition \ref{P2}, we can also obtain higher integrability for $\frac{3}{2} < \gamma \le 3$.

\begin{prop}
\label{P3}
Let $(\rho, \bu)$ be a weak solution of \eqref{E11} and \eqref{E12} in $\bR^3$ with $p(\rho) = \rho^{\gamma}$ with $\frac{3}{2} < \gamma \le 3$ and $G \in D_{0}^{1,\frac{6\gamma}{5\gamma-6}}(\bR^3)$.
Suppose
\[
\sup_{R > 1} R^{-\alpha} \norm{\rho}_{L^{\gamma}(B(R))} < \infty
\]
for some 
$0 \le \alpha \le \frac{3}{\gamma}$.
Then we have the following:
\begin{itemize}
\item
If $\frac{3}{2(\gamma-1)} \le \alpha \le \frac{3}{\gamma}$, then
\[
\sup_{R > 1} R^{- \left( \alpha - \frac{1}{\gamma} \right)} \norm{\rho}_{L^{3\gamma-3}(B(R))} < \infty.
\]
\item
If $0 \le \alpha < \frac{3}{2(\gamma-1)}$, then
\[
\sup_{R > 1} R^{- \left( \frac{\alpha}{\gamma} + \frac{1}{2\gamma} \right)} \norm{\rho}_{L^{3\gamma-3}(B(R))} < \infty.
\]
\end{itemize}
\end{prop}

\begin{proof}
Taking $\liminf_{n \to \infty}$ on \eqref{E519} and \eqref{E519b}, we get desired results.

\end{proof}

Similar to Corollary \ref{C1} and Corollary \ref{C11}, we also consider the case when $\rho$ belongs to $L^{\gamma}(\bR^3)$.

\begin{corollary}
\label{C2}
Let $(\rho, \bu)$ be a weak solution of \eqref{E11} and \eqref{E12} in $\bR^3$ with $p(\rho) = \rho^{\gamma}$ with $3 \le \gamma < 6$ and $G \in D_{0}^{1,\frac{6\gamma}{5\gamma-6}}(\bR^3)$ .
If $\rho \in L^{\gamma}(\bR^3)$, then $\rho \in L^{2\gamma}(\bR^3)$.
\end{corollary}

\begin{proof}
\textbf{Step 1)}
As in the proof of Proposition \ref{P2}, for $3 < \gamma < 6$, we choose $k \in \N$ such that
\[
\frac{9 - 3 \cdot \left( \frac{2}{3} \right)^{k}}{3 - 2 \cdot \left( \frac{2}{3} \right)^k} \le \gamma < \frac{9 - 3 \cdot \left( \frac{2}{3} \right)^{k-1}}{3 - 2 \cdot \left( \frac{2}{3} \right)^{k-1}}.
\]
Recall the definitions
\[
\widetilde \gamma_n  =  (2\gamma -3) \left( 1 - \left( \frac{2}{3} \right)^{n} \right)
\qand
c_n = \gamma + \widetilde \gamma_n 
\]
For $R_n = \left( 1 + \frac{1}{2^n}\right)R$, we have defined $d_n$ as a minimum number satisfying 
\[
\norm{\rho}_{L^{c_n}(B(R_{n-1}))} \lesssim R^{d_n}.
\]
By obtaining $\rho \in L^{c_k}(\bR^3)$, we shall first obtain $\rho \bu \otimes \bu \in L^2(\bR^3)$.
As $|f^+ - g^+| \le |f-g|$, note that
\[
|(\rho-1)^+ - (\rho^{\ep}-1)^+ | \le |\rho - \rho^{\ep}|.
\]
Since $\rho \in L^{\gamma}(\bR^3)$, we have
\[
(\rho-1)^+ \in L^1 \cap L^{\gamma} (\bR^3)
\]
and hence
\[
(\rho^{\ep}-1)^+ \to (\rho-1)^+
\]
a.e. and in $L^p$ for all $1 \le p \le \gamma$. 
For $\rho \ge 0$, $R > 0$, and $\gamma, \alpha \ge 0$, we have
\begin{align*}
\int \rho^{\gamma} (\rho^{\ep})^{\alpha} \varphi_{R_n,R_{n-1}}^2 dx
&= \int \rho^{\gamma} ((\rho^{\ep})^{\alpha} - 1) \varphi_{R_n,R_{n-1}}^2 dx
+ \int \rho^{\gamma} \varphi_{R_n,R_{n-1}}^2 dx \\
&\le \int_{\set{0 \le \rho^{\ep} < 1}} \rho^{\gamma} ((\rho^{\ep})^{\alpha} - 1) \varphi_{R_n,R_{n-1}}^2 dx
+ \int_{\set{1 \le \rho^{\ep} < 2}} \rho^{\gamma} ((\rho^{\ep})^{\alpha} - 1) \varphi_{R_n,R_{n-1}}^2 dx \\
&\quad + \int_{\set{2 \le \rho^{\ep}}} \rho^{\gamma} ((\rho^{\ep})^{\alpha} - 1) \varphi_{R_n,R_{n-1}}^2 dx
+ \int \rho^{\gamma} \varphi_{R_n,R_{n-1}}^2 dx \\
&\le 2^{\alpha} \left( \int \rho^{\gamma} [(\rho^{\ep} -1)^{\alpha}]^{+} \varphi_{R_n,R_{n-1}}^2 dx 
+ \int \rho^{\gamma} \varphi_{R_n,R_{n-1}}^2 dx \right)
\end{align*}
where we used $((\rho^{\ep})^{\alpha} - 1) \le 2^{\alpha} (\rho^{\ep} - 1)^{\alpha} = 2^{\alpha} [(\rho^{\ep} - 1)^{\alpha}]^{+}$ in the last inequality, which holds for $\rho^{\ep} \ge 2$.
Therefore, we get
\[
\int \rho^{\gamma} (\rho^{\ep})^{\alpha} \varphi_{R_n,R_{n-1}}^2 dx
\lesssim \int \rho^{\gamma} [(\rho^{\ep} -1)^{\alpha}]^{+} \varphi_{R_n,R_{n-1}}^2 dx 
+ \norm{\rho}_{L^\gamma(\bR^3)}^{\gamma}.
\]

\textbf{Step 2)}
Define
\[
\bfphi_{\ep, \alpha, n} := \nabla \Delta^{-1} \left( [(\rho^{\ep} -1)^{\alpha}]^{+} \varphi_{R_n,R_{n-1}} \right) \varphi_{R_n,R_{n-1}}. 
\]
Then $\bfphi_{\ep, \alpha, n} \in C^{\infty}_c(\bR^3; \bR^3)$ and
\[
\divg \bfphi_{\ep, \alpha, n} 
= [(\rho^{\ep}-1)^{\alpha}]^{+} \varphi_{R_n, R_{n-1}}^2
+ \nabla \Delta^{-1} \left( [(\rho^{\ep} -1)^{\alpha}]^{+} \varphi_{R_n,R_{n-1}} \right) \cdot \nabla \varphi_{R_n,R_{n-1}}. 
\]
Since $\nabla \Delta^{-1}$ maps $L^p(\bR^3)$ to $D^{1,p}_0(\bR^3)$, we have, by means of Jensen's inequality,
\begin{align}
\begin{split}
\label{E522}
\norm{\nabla \bfphi_{\ep, \alpha, n}}_{L^{p}(B(R_{n-1}))} 
&\lesssim \norm{[(\rho^{\ep}-1)^{\alpha}]^{+} }_{L^{p}(B(R_{n-1}))}
+ R^{-1} \norm{[(\rho^{\ep}-1)^{\alpha}]^{+} }_{L^{\frac{3p}{p+3}}(B(R_{n-1}))} \\
&\lesssim \norm{[(\rho^{\ep}-1)^{\alpha}]^{+} }_{L^{p}(B(R_{n-1}))}
\end{split}
\end{align}
for $1 < p < \infty$.
Using $\bfphi_{\ep, \widetilde \gamma_n , n} $ as a test function in the momentum equation \eqref{E16}, we have
\begin{align*}
\int \rho^{\gamma} [(\rho^{\ep}-1)^{\widetilde \gamma_n }]^{+} \varphi_{R_n, R_{n-1}}^2 dx
&= - \int_{B(R_{n-1})} \rho^{\gamma}  \nabla \Delta^{-1} \left( [(\rho^{\ep} -1)^{\widetilde \gamma_n }]^{+} \varphi_{R_n,R_{n-1}} \right) \cdot \nabla \varphi_{R_n,R_{n-1}} dx \\
&\quad + \int_{B(R_{n-1})} \mathbb S(\nabla \bu) : \nabla \bfphi_{\ep, \widetilde \gamma_n , n}  dx
-\int_{B(R_{n-1})} \rho \nabla G \cdot \bfphi_{\ep, \widetilde \gamma_n , n} dx \\
&\quad - \int_{B(R_{n-1})} \rho \bu \otimes \bu : \nabla \bfphi_{\ep, \widetilde \gamma_n , n} dx \\
&:= L_1 + L_2 + L_3 + L_4.
\end{align*}
Note that 
\[
\norm{\rho}_{L^{2\gamma}(B(R))} \lesssim R^{\frac{1}{2\gamma^2}} 
\]
by \eqref{E511b} with $\alpha = 0$, so we have
\[
\norm{\rho}_{L^{\frac{6\gamma}{5}}(B(R))} \lesssim R^{\frac{1}{6\gamma^2}}. 
\]
By H\"older's inequality, 
\begin{align*}
L_1 
&\lesssim R^{-1} \norm{\rho}_{L^{\frac{6\gamma}{5}}(B(R_{n-1}))}
\norm{ \nabla \Delta^{-1} \left( [(\rho^{\ep} -1)^{\widetilde \gamma_n }]^{+} \varphi_{R_n,R_{n-1}} \right)}_{L^6(B(R_{n-1}))} \\
&\lesssim R^{-1} \norm{\rho}_{L^{\frac{6\gamma}{5}}(B(R_{n-1}))}
\norm{[(\rho^{\ep} -1)^{\widetilde \gamma_n }]^{+} }_{L^2(B(R_{n-1}))} \\
&\lesssim R^{-1 + \frac{1}{6\gamma^2}} \norm{(\rho^{\ep} -1)^{+} }_{L^{2\widetilde \gamma_n}(B(R_{n-1}))}^{\widetilde \gamma_n}.
\end{align*}
By H\"older's inequality, Sobolev's inequality, and \eqref{E522}, we can obtain
\begin{align*}
L_2 + L_3
&\lesssim \norm{\nabla \bu}_{L^2(\bR^3)} 
\norm{\nabla \bfphi_{\ep, \widetilde \gamma_n, n}}_{L^{2}(B(R_{n-1}))} \\
&\quad + \norm{\rho}_{L^{\gamma} \left(B(R_{n-1})\right)} 
\norm{\nabla G}_{L^{\frac{6\gamma}{5\gamma-6}}(\bR^3)}
\norm{\bfphi_{\ep, \widetilde \gamma_n, n}}_{L^{6}(B(R_{n-1}))} \\
&\lesssim \norm{\nabla \bfphi_{\ep, \widetilde \gamma_n, n}}_{L^{2}(B(R_{n-1}))} \\
&\lesssim \norm{(\rho^{\ep}-1)^{+} }_{L^{2\widetilde \gamma_n}(B(R_{n-1}))}^{\widetilde \gamma_n}.
\end{align*}
Here, we also use the assumption $\rho \in L^{\gamma}(\bR^3)$, $G \in D^{1, \frac{6\gamma}{5\gamma-6}}_0(\bR^3)$, and $\bu \in D^{1,2}_0(\bR^3)$.
Similarly, 
\begin{align*}
L_4
&\lesssim \norm{\rho}_{L^{c_{n-1}} \left(B(R_{n-1})\right)} 
\norm{\nabla \bu}_{L^2(\bR^3)}^2 
\norm{\nabla \bfphi_{\ep, \widetilde \gamma_n, n}}_{L^{\frac{3c_{n-1}}{2c_{n-1}-3}} \left(B(R_{n-1})\right)} \\
&\lesssim \norm{\rho}_{L^{c_{n-1}} \left(B(R_{n-1})\right)}
\norm{(\rho^{\ep}-1)^{+} }_{L^{c_{n-1}}(B(R_{n-1}))}^{\widetilde \gamma_n}.
\end{align*}

\textbf{Step 3)}
Combining all the estimates, 
\[
\int \rho^{\gamma} (\rho^{\ep})^{\widetilde \gamma_n } \varphi_{R_n,R_{n-1}}^2 dx
\lesssim \norm{(\rho^{\ep} -1)^{+} }_{L^{2\widetilde \gamma_n}(B(R_{n-1}))}^{\widetilde \gamma_n}
+ R^{d_{n-1}} \norm{(\rho^{\ep}-1)^{+} }_{L^{c_{n-1}}(B(R_{n-1}))}^{\widetilde \gamma_n}.
\]
Taking $\liminf_{\ep \to 0}$ both sides, we obtain
\begin{align*}
\int \rho^{c_{n}} \varphi_{R_n,R_{n-1}}^2 dx
&\lesssim \norm{(\rho -1)^{+} }_{L^{2\widetilde \gamma_n}(B(R_{n-1}))}^{\widetilde \gamma_n}
+ R^{d_{n-1}} \norm{(\rho-1)^{+} }_{L^{c_{n-1}}(B(R_{n-1}))}^{\widetilde \gamma_n} \\
&\lesssim \norm{(\rho -1)^{+} }_{L^{2\widetilde \gamma_n}(B(R_{n-1}))}^{\widetilde \gamma_n}
+ R^{d_{n-1}} \norm{\rho }_{L^{c_{n-1}}(B(R_{n-1}))}^{\widetilde \gamma_n}.
\end{align*}
Since $2\widetilde \gamma_n \le c_{n-1}$ for $n \le k$, we have
\[
\int_{B(R_{n-1})} \rho^{c_{n}} dx
\le \int \rho^{c_{n}} \varphi_{R_n,R_{n-1}}^2 dx
\lesssim R^{(1+ \widetilde \gamma_n) d_{n-1}}. 
\]
That is, 
\begin{equation}
\label{E521}
c_{n} d_{n} = (1+ \widetilde \gamma_n) d_{n-1}.
\end{equation}
As $\rho \in L^{\gamma}(\bR^3)$, we have $d_0=0$ and hence
\[
\rho \in L^{c_n}(\bR^3)
\]
for all $n \le k$.
Since $c_k \ge 6$, we can conclude that $\rho \bu \otimes \bu \in L^2(\bR^3; \bR^{3 \times 3})$.
Therefore, following the argument of Corollary \ref{C11}, we can easily conclude that
\[
\rho \in L^{2\gamma}(\bR^3).
\]
For $\gamma=3$ case, we can similarly obtain the result by Proposition \ref{P3} instead of \eqref{E511b} and considering $\lim_{n \to \infty} d_n$ in \eqref{E521}.
\end{proof}

\subsection{Proof of Theorem \ref{T3}}
\label{S6}

From Proposition \ref{P1}, Proposition \ref{P2}, and Proposition \ref{P3}, we have
\begin{equation}
\label{E61}
\sup_{R > 1} R^{- f(\alpha, \gamma)} \norm{\rho}_{L^{2\gamma}(B(R))} < \infty
\end{equation}
and hence $\rho \in L^{2\gamma}_{\loc}(\bR^3)$.
So we may use $\bu \varphi_{R,2R}$ as a test function in  \eqref{E16}  to get
\begin{equation}
\label{E62}
\int_{\bR^3} \rho \bu \otimes \bu : \nabla \left(\bu \varphi_{R,2R} \right) + \rho^\gamma \divg \left(\bu \varphi_{R,2R} \right) dx 
= \int_{\bR^3} \mathbb S(\nabla \bu) : \nabla \left(\bu \varphi_{R,2R} \right) dx 
- \int_{\bR^3} \rho \nabla G \cdot \left(\bu \varphi_{R,2R} \right) dx.
\end{equation}
Direct calculation using integration by parts yields
\begin{align*}
&\int_{\bR^3} \left( \mu |\nabla \bu|^2 + (\mu + \lambda) |\divg \bu|^2 \right) \varphi_{R,2R} dx \\
&= \frac{\mu}{2} \int_{\bR^3} |\bu|^2 \Delta \varphi_{R, 2R} dx
+ \int_{\bR^3} \mu (\bu \otimes \bu) : \nabla^2 \varphi_{R, 2R} + (\mu + \lambda) \divg \bu (\bu \cdot \nabla \varphi_{R,2R}) dx \\
&\quad + \frac{1}{2} \int_{\bR^3} \rho \bu \cdot \nabla \left( |\bu|^2 \varphi_{R, 2R} \right) dx
+ \frac{1}{2} \int_{\bR^3} \rho |\bu|^2 \bu \cdot \nabla \varphi_{R, 2R} dx \\
&\quad + \int_{\bR^3} \rho^{\gamma} \divg \bu \varphi_{R, 2R} dx
+ \int_{\bR^3} \rho^{\gamma} \bu \cdot \nabla \varphi_{R, 2R} dx
+ \int_{\bR^3} \rho \nabla G \cdot \left(\bu \varphi_{R,2R} \right) dx.
\end{align*}
Since $\rho \bu \in L^{3}_{\loc}(\bR^3)$, we have from \eqref{E15} that
\[
\int_{\bR^3} \rho \bu \cdot \nabla \left( |\bu|^2 \varphi_{R, 2R} \right) dx = 0.
\]
Regularizing \eqref{E15}, we also have
\[
\divg (\rho^{\ep} \bu) = r_{\ep} \quad \text{a.e. in} \quad \bR^3
\]
where
\[
r_{\ep} = \divg (\rho^{\ep} \bu) - \divg((\rho \bu)^{\ep}).
\]
By Lemma \ref{L21}, we have, for $1 \le q \le \frac{2\gamma}{\gamma+1}$,
\begin{equation}
\label{E63}
r_{\ep} \to 0 \quad \text{strongly in} \quad L^{q}_{\loc}(\bR^3).
\end{equation}
Performing integration by parts, we obtain
\begin{align*}
\int_{\bR^3} \rho^{\gamma} \divg \bu \varphi_{R, 2R} dx
&= \int_{\bR^3} \left( \rho^{\gamma} - (\rho^{\ep})^{\gamma} \right) \divg \bu \varphi_{R, 2R} dx
+ \int_{\bR^3} (\rho^{\ep})^{\gamma} \divg \bu \varphi_{R, 2R} dx \\
&= \int_{\bR^3} \left( \rho^{\gamma} - (\rho^{\ep})^{\gamma} \right) \divg \bu \varphi_{R, 2R} dx
- \int_{\bR^3} \bu \cdot \nabla \left( (\rho^{\ep})^{\gamma} \varphi_{R, 2R} \right) dx.
\end{align*}
Performing integration by parts to the last term, we get
\[
\int_{\bR^3} \bu \cdot \nabla \left( (\rho^{\ep})^{\gamma} \varphi_{R, 2R} \right) dx
= - \frac{\gamma}{\gamma-1} \int_{\bR^3} r_{\ep} (\rho^{\ep})^{\gamma-1} \varphi_{R, 2R} dx
- \frac{1}{\gamma-1} \int_{\bR^3} (\rho^{\ep})^{\gamma} \bu \cdot \nabla \varphi_{R, 2R} dx.
\]
Lastly, we have
\[
\int_{\bR^3} \rho \nabla G \cdot \left(\bu \varphi_{R,2R} \right) dx 
= \int_{\bR^3} \left( \rho - \rho^{\ep} \right) \nabla G \cdot \left(\bu \varphi_{R,2R} \right) dx 
+ \int_{\bR^3} \rho^{\ep} \nabla G \cdot \left(\bu \varphi_{R,2R} \right) dx. 
\]
Performing integration by parts to the last term,
\[
\int_{\bR^3} \rho^{\ep} \nabla G \cdot \left(\bu \varphi_{R,2R} \right) dx
= - \int_{\bR^3} G r_{\ep} \varphi_{R, 2R} dx  
- \int_{\bR^3} G \rho^{\ep} \bu \cdot \nabla \varphi_{R,2R} dx.
\]
Combining all the calculations, we get the following local energy inequality by H\"older's inequality
\begin{align}
\begin{split}
\label{E64}
\int_{B(R)} |\nabla \bu|^2 dx 
&\lesssim R^{-2} \norm{\bu}_{L^2(A_R)}^2
+ \norm{\nabla \bu}^2_{L^2(A_R)} 
+ R^{-1} \norm{\rho}_{L^{2\gamma}(A_R)} \norm{\bu}^3_{L^{\frac{6\gamma}{2\gamma-1}}(A_R)} \\
&\quad + R^{-1} \norm{\rho}_{L^{2\gamma}(A_R)}^{\gamma} \norm{\bu}_{L^2(A_R)} 
+ \norm{\rho^{\gamma} - (\rho^{\ep})^{\gamma}}_{L^2(B(2R))} \norm{\nabla \bu}_{L^2(B(2R))} \\
&\quad + \norm{r_{\ep}}_{L^{\frac{2\gamma}{\gamma+1}}(B(2R))} \norm{\rho^{\ep}}^{\gamma-1}_{L^{2\gamma}(B(2R))} 
+ \norm{\rho - \rho^{\ep}}_{L^{\gamma}(B(2R))} \norm{\nabla G}_{L^{\frac{6\gamma}{5\gamma-6}}(B
(2R))} \norm{\nabla \bu}_{L^2(\bR^3)} \\
&\quad + \norm{G}_{L^{\frac{2\gamma}{\gamma-2}}(B(2R))} \norm{r_{\ep}}_{L^{\frac{2\gamma}{\gamma+2}}(B(2R))}
+ R^{-1} \norm{G}_{L^{\frac{2\gamma}{\gamma-2}}(A_R)} \norm{\rho^{\ep}}_{L^{\gamma}(A_R)} \norm{\bu}_{L^{2}(A_R)}
\end{split}
\end{align}
where $A_R = B(2R) \setminus B(R)$.
Using the mean value theorem, we have
\begin{equation}
\label{E65}
\norm{\rho^{\gamma} - (\rho^{\ep})^{\gamma}}_{L^2(B(2R))} 
\lesssim \norm{\rho + \rho^{\ep}}_{L^{2\gamma}(B(2R))}^{\gamma-1}
\norm{\rho - \rho^{\ep}}_{L^{2\gamma}(B(2R))}.
\end{equation}
From \eqref{E63}, \eqref{E65}, the properties of mollifier, and the assumption $G \in D^{1, \frac{6\gamma}{5\gamma-6}}_0(\bR^3)$, we have
\begin{align}
\begin{split}
\label{E66}
\int_{B(R)} |\nabla u|^2 dx 
&\lesssim R^{-2} \norm{\bu}_{L^2(A_R)}^2
+ \norm{\nabla \bu}^2_{L^2(A_R)} 
+ R^{-1} \norm{\rho}_{L^{2\gamma}(A_R)} \norm{\bu}^3_{L^{\frac{6\gamma}{2\gamma-1}}(A_R)} \\
&\quad + R^{-1} \norm{\rho}_{L^{2\gamma}(A_R)}^{\gamma} \norm{\bu}_{L^2(A_R)}
+ R^{-1} \norm{\rho}_{L^{\gamma}(A_R)} \norm{\bu}_{L^{2}(A_R)} \\
&:= \sum_{i=1}^5 J_i
\end{split}
\end{align}
by letting $\ep \to 0$ in \eqref{E64}.
Since $\bu \in D^{1,2}_0(\bR^3; \bR^3)$, Jensen's inequality yields
\[
J_1 + J_2 \lesssim \norm{\bu}_{L^6(A_R)} + \norm{\nabla \bu}_{L^2(A_R)} \to 0
\]
as $R \to \infty$.
For $J_3$, $J_4$ and $J_5$, we use H\"older's inequality to get
\begin{align*}
J_3 &\lesssim R^{-1} \norm{\rho}_{L^{2\gamma}(A_R)}
 \norm{\bu}^{\frac{3\beta(\gamma-1)}{\gamma(6-\beta)}}_{L^{\beta}(A_R)} 
\norm{\bu}^{3-\frac{3\beta(\gamma-1)}{\gamma(6-\beta)}}_{L^6(A_R)}, \\
J_4 &\lesssim R^{-1} \norm{\rho}_{L^{2\gamma}(A_R)}^{\gamma} 
\norm{\bu}_{L^\beta(A_R)}^{\frac{2\beta}{6-\beta}} 
\norm{\bu}_{L^6(A_R)}^{\frac{6-3\beta}{6-\beta}}, \\
J_5 &\lesssim R^{-1} \norm{\rho}_{L^{\gamma}(A_R)} 
\norm{\bu}_{L^\beta(A_R)}^{\frac{2\beta}{6-\beta}} 
\norm{\bu}_{L^6(A_R)}^{\frac{6-3\beta}{6-\beta}}.
\end{align*}
Note that $J_5$ is dominated by $J_4$ since $\gamma f(\alpha, \gamma) \ge \alpha$ and $J_4$ is dominated by $J_5$ if and only if $f(\alpha, \gamma) \le \frac{\gamma-3}{\gamma(3\gamma-5)}$.
Using \eqref{E110}, \eqref{E111}, and \eqref{E61}, and the assumptions $\bu \in D^{1,2}_0(\bR^3; \bR^3)$, we obtain
\[
\liminf_{R \to \infty} \left( J_3 + J_4 + J_5 \right) = 0.
\]
Therefore, we can conclude from \eqref{E66} that $\nabla \bu = 0$ a.e. and hence $\bu = 0$ a.e.

\subsection{Proof of Theorem \ref{T21}}
\label{S9}

Since we assumed $\rho \in L^{6} \cap L^{\gamma}(\bR^3)$, we can obtain from Corollary \ref{C1} and Corollary \ref{C2} that $\rho \in L^{6} \cap L^{2\gamma}(\bR^3)$.
Combine this with the assumption on $\bu$ and $G$, we may use $\bu$ as a test function in \eqref{E16} to get
\begin{equation}
\label{E81}
\int_{\bR^3} \rho \bu \otimes \bu : \nabla \bu dx 
+ \int_{\bR^3} \rho^\gamma \divg \bu dx 
= \int_{\bR^3} \mathbb S(\nabla \bu) : \nabla \bu dx 
- \int_{\bR^3} \rho \nabla G \cdot \bu dx.
\end{equation}
Then the Diperna--Lions theory yields, see e.g. \cite[Lemma 3.3]{MR2084891}:   
\[
\int_{\bR^3} \rho^{\gamma} \bu \cdot \nabla \phi dx
+ \int_{\bR^3} (\gamma-1) \rho^{\gamma} \divg \bu \phi dx 
= 0
\]
for all $\phi \in C^{\infty}_c (\bR^3)$.
Since $\rho^{\gamma} \bu \in L^1(\bR^3; \bR^3)$ and $\rho^{\gamma} \divg \bu \in L^1(\bR^3)$, we have
\[
\int_{\bR^3} \rho^\gamma \divg \bu dx = 0.
\]
Hence
\[
\int_{\bR^3} \mathbb S(\nabla \bu) : \nabla \bu dx
= \int_{\bR^3} \rho \bu \otimes \bu : \nabla \bu dx 
+ \int_{\bR^3} \rho \nabla G \cdot \bu dx.
\]
Note that $\rho \bu \in L^{3}\cap L^{\frac{6\gamma}{\gamma+3}}(\bR^3)$.
Hence, we have from \eqref{E15} that
\[
\int_{\bR^3} \rho \bu \otimes \bu : \nabla \bu dx 
+ \int_{\bR^3} \rho \nabla G \cdot \bu dx
= \int_{\bR^3} \rho \bu \cdot \nabla \left( \frac{1}{2} |\bu|^2 + G \right) dx = 0.
\]
Therefore, we have
\[
\int_{\bR^3} \mu |\nabla \bu|^2 + (\mu/3 + \eta) |\divg \bu|^2 dx = 0
\]
and hence $\bu = 0$ a.e. as desired.

\subsection{Proof of Theorem \ref{T22}}
\label{S10}

Proof of Theorem \ref{T22} is analogous to that of Theorem \ref{T21}.
By Corollary \ref{C11}, we have 
\[
\rho \in L^{2\gamma}(\bR^3).
\]
Using \eqref{E112}, we have
\[
\rho \bu \in L^{\frac{2\gamma}{\gamma+1}} \cap L^{\frac{6\gamma}{\gamma+3}}(\bR^3; \bR^3), 
\quad \rho \bu \otimes \bu \in L^1 \cap L^{\frac{6\gamma}{3+2\gamma}}(\bR^3; \bR^{3 \times 3}), 
\qand \rho \nabla G \in L^\frac{6}{5} \cap L^{\frac{6\gamma}{5\gamma-3}}(\bR^3; \bR^3).
\] 
Note that $\frac{2\gamma}{\gamma+1} \le 3 \le \frac{6\gamma}{\gamma+3}$ and $\frac{6\gamma}{3+2\gamma} \ge 2$ since $\gamma \ge 6$.
Testing $\bu$ to \eqref{E16} and using Diperna--Lions theory and \eqref{E15}, we can conclude that $\bu = 0$ a.e.

\subsection{Proof of Theorem \ref{T23}}
\label{S11}

Proof of Theorem \ref{T23} is also analogous to that of Theorem \ref{T21}.
By Corollary \ref{C2}, we have 
\[
\rho \in L^{2\gamma}(\bR^3).
\]
This time, we have
\[
\rho \bu \in L^{\frac{6\gamma}{\gamma+6}} \cap L^{\frac{6\gamma}{\gamma+3}}(\bR^3; \bR^3), 
\quad \rho \bu \otimes \bu \in L^{\frac{3\gamma}{\gamma+3}} \cap L^{\frac{6\gamma}{3+2\gamma}}(\bR^3; \bR^{3 \times 3}), 
\qand \rho \nabla G \in L^\frac{6}{5} \cap L^{\frac{6\gamma}{5\gamma-3}}(\bR^3; \bR^3).
\] 
Note that $\frac{6\gamma}{\gamma+6} < 3 \le \frac{6\gamma}{\gamma+3}$ and $\frac{3\gamma}{\gamma+3} < 2 \le \frac{6\gamma}{3+ 2\gamma}$ since $3 \le \gamma < 6$.
We may repeat the same argument to conclude that $\bu = 0$ a.e.


\section*{Acknowledgement}

E. F. was supported by the Czech Science Foundation (GA\v CR) (No. 24--11034S). The Institute of Mathematics of the Academy of Sciences of
the Czech Republic is supported by RVO:67985840. E.F. is a member of the Ne\v cas Center for Mathematical Modelling.
Y. C. and M. Y. have been supported by the National Research Foundation of Korea(NRF) grant funded by the Korean government(MSIT) (No. 2021R1A2C4002840).


\end{document}